\documentclass{amsart}
\usepackage[letterpaper,margin=1in]{geometry}
\usepackage{amsthm,amsmath,amssymb,enumerate,amsfonts,mathtools,graphicx}
\usepackage[utf8]{inputenc}
\usepackage{ytableau}
\usepackage{array,booktabs}

\ytableausetup{boxsize=1.25em}

\newcommand{\RSyt}[1]{%
  \raisebox{-.45\height}{\(\ytableaushort{#1}\)}%
}
\usepackage{tikz}
\usepackage{titletoc}
\usepackage{xcolor}
\usetikzlibrary{arrows,decorations.pathmorphing}
\usetikzlibrary{arrows.meta,positioning,fit,calc,matrix,decorations.pathreplacing}
\usepackage[colorlinks=true,linkcolor=blue,citecolor=blue,urlcolor=blue]{hyperref}
\usepackage{aliascnt}
\usepackage[nameinlink,capitalise,noabbrev]{cleveref}

\makeatletter
\renewcommand{\l@section}{%
  \@tocline{1}{6pt plus2pt}{0pt}{2.8em}{\bfseries}%
}
\renewcommand{\l@subsection}{%
  \@tocline{2}{2pt plus1pt}{2.8em}{3.6em}{\small}%
}
\makeatother

\newtheorem{theorem}{Theorem}[section]
\crefname{theorem}{Theorem}{Theorems}
\Crefname{theorem}{Theorem}{Theorems}

\newaliascnt{defprop}{theorem}
\newtheorem{def-prop}[defprop]{Definition-Proposition}
\aliascntresetthe{defprop}
\crefname{defprop}{Definition-Proposition}{Definition-Propositions}
\Crefname{defprop}{Definition-Proposition}{Definition-Propositions}

\newaliascnt{proposition}{theorem}
\newtheorem{proposition}[proposition]{Proposition}
\aliascntresetthe{proposition}
\crefname{proposition}{Proposition}{Propositions}
\Crefname{proposition}{Proposition}{Propositions}

\newaliascnt{conj}{theorem}

\aliascntresetthe{conj}
\crefname{conj}{Conjecture}{Conjectures}
\Crefname{conj}{Conjecture}{Conjectures}

\newaliascnt{lemma}{theorem}
\newtheorem{lemma}[lemma]{Lemma}
\aliascntresetthe{lemma}
\crefname{lemma}{Lemma}{Lemmas}
\Crefname{lemma}{Lemma}{Lemmas}

\newaliascnt{cor}{theorem}
\newtheorem{cor}[cor]{Corollary}
\aliascntresetthe{cor}
\crefname{cor}{Corollary}{Corollaries}
\Crefname{cor}{Corollary}{Corollaries}

\theoremstyle{definition}
\newaliascnt{ex}{theorem}
\newtheorem{ex}[ex]{Example}
\aliascntresetthe{ex}
\crefname{ex}{Example}{Examples}
\Crefname{ex}{Example}{Examples}

\newaliascnt{quest}{theorem}
\newtheorem{quest}[quest]{Question}
\aliascntresetthe{quest}
\crefname{quest}{Question}{Questions}
\Crefname{quest}{Question}{Questions}

\newaliascnt{defin}{theorem}
\newtheorem{defin}[defin]{Definition}
\aliascntresetthe{defin}
\crefname{defin}{Definition}{Definitions}
\Crefname{defin}{Definition}{Definitions}

\theoremstyle{remark}
\newtheorem*{remark}{Remark}

\makeatletter
\newcommand{\Rmnum}[1]{\expandafter\@slowromancap\romannumeral #1@}
\makeatother

\newcommand{\Ext}{\operatorname{Ext}}

\newcommand{\Cat}{\operatorname{Cat}}
\newcommand{\wt}{\operatorname{wt}}
\newcommand{\Ttilde}{\widetilde{T}}

\newcommand{\F}{\mathbb{F}}

\newcommand{\Y}{\mathscr{Y}}

\newcommand{\std}{\operatorname{Std}}

\renewcommand{\k}{\mathbf{k}}

\newcommand{\GL}{\operatorname{GL}}

\newcommand{\sgn}{\operatorname{sgn}}

\DeclareMathOperator{\Dem}{Dem}
\DeclareMathOperator{\Red}{Red}
\DeclareMathOperator{\pos}{pos}
\DeclareMathOperator{\Supp}{Supp}
\DeclareMathOperator{\Fwd}{Fwd}
\DeclareMathOperator{\Bwd}{Bwd}

\makeatletter
\newcommand{\affiliations}[1]{\gdef\@affiliations{#1}}
\let\@affiliations\@empty
\def\@setauthors{%
  \begingroup
  \def\thanks{\protect\thanks@warning}%
  \trivlist
  \centering\footnotesize \@topsep30\p@\relax
  \advance\@topsep by -\baselineskip
  \item\relax
  \author@andify\authors
  \def\\{\protect\linebreak}%
  \MakeUppercase{\authors}%
  \ifx\@empty\contribs
  \else
    ,\penalty-3 \space \@setcontribs
    \@closetoccontribs
  \fi
  \ifx\@empty\@affiliations
  \else
    \par\vskip.45em
    {\normalfont\scriptsize\@affiliations\par}%
  \fi
  \endtrivlist
  \endgroup
}
\makeatother

\title{Random permutations from Bott--Samelson varieties}
\author[J. Li]{Jingqi Li\({}^{1}\)}

\author[H. Yin]{Haorun Yin\({}^{2}\)}

\author[W. Yu]{Wenbin Yu\({}^{3}\)}

\author[S. Zeng]{Shixuan Zeng\({}^{4}\)}

\affiliations{%
  \begin{tabular}{c}
  \textsuperscript{1} Fudan University, Shanghai, China; \texttt{leejingqi3@gmail.com}\\
  \textsuperscript{2} China Agricultural University, Beijing, China; \texttt{archangelyhr@gmail.com}\\
  \textsuperscript{3} Guangdong Technion--Israel Institute of Technology, Shantou, Guangdong, China\\
  \texttt{yu24927@gtiit.edu.cn}\\
  \textsuperscript{4} University of Science and Technology of China, Hefei, Anhui, China; \texttt{zsx0515@mail.ustc.edu.cn}
  \end{tabular}}

\thanks{The authors are listed alphabetically. Jingqi Li and Shixuan Zeng made the principal contributions to the formulation of the main results and the proofs.}

\begin{document}
\begin{abstract}
Motivated by the random pipe dream model of \cite{MPPY2025}, we study a family of probability distributions on \(S_n\) arising from Bott--Samelson varieties over finite fields. More precisely, for a word \(R\), we consider the Bott--Samelson map \(\pi_R:\mathrm{BS}^R\to \mathcal F\ell_n\) and define a distribution \(\mathbb P_{R,q}\) by counting the \(\mathbb F_q\)-points in the inverse images of Schubert cells. For a suitable choice of parameters \(p_1=q/(1+q)\) and \(p_2=1/q\), this construction recovers a special case of the random pipe dream distribution.

The main problem considered in this note is to determine which combinatorial properties of a reduced word are detected by the distribution \(\mathbb P_{R,q}\). We prove the stronger statement that, for arbitrary reduced words \(R_1,R_2\), the equality
\[
\mathbb P_{R_1,q}=\mathbb P_{R_2,q}
\]
as functions of \(q\) holds if and only if \(R_1\) and \(R_2\) lie in the same commutation class. In particular, equality of distributions already forces the two words to represent the same permutation. The proof combines the Bott--Samelson interpretation with Demazure products, commutation-class invariants, and Hecke-algebraic arguments.
\end{abstract}
\maketitle

\tableofcontents

\section{Introduction}\label{sec:introduction}

Schubert polynomials, $\{\mathfrak{S}_w:w\in S_n\}$, introduced by Lascoux--Sch\"utzenberger \cite{LS82}, are polynomial representatives of cohomology classes of Schubert varieties in the complete flag variety $\mathcal{F} \ell_n$. These polynomials, as well as their various generalizations, are central objects in the study of Schubert calculus. In \cite{BB93}, Bergeron--Billey introduced \textbf{pipe dreams} as a combinatorial model for Schubert polynomials. 
Since their introduction, pipe dreams have seen numerous applications and connections in algebraic combinatorics; see, e.g., \cite{KM04,KY04,KM05,KoM05,WY12,KST12,MPP19,G21,GH23,HS24,PSW24} and references therein.


In \cite{MPPY2025}, Morales--Panova--Petrov--Yeliussizov studied asymptotic behaviors of random pipe dreams, motivated by the question of Stanley [\cite{Sta17b}, Section~5] on principal specializations of Schubert polynomials. They introduced a two-parameter probability distribution, which we denote by
\[
\mathbb P^{\mathrm{pd}}_{p_1,p_2} :S_n \longrightarrow [0,1],
\]
arising from a random pipe dream together with a random reduction procedure. They gave explicit conjectures about the probability distribution of the resulting permutation of a random pipe dream. 

In this note, we provide an alternative framework, using Bott--Samelson varieties, for studying special cases of random pipe dreams. Specifically, we start by observing that when
\[
p_1=\frac{q}{1+q},
\qquad
p_2=\frac{1}{q},
\]
the above random pipe dream distribution admits an algebraic realization in terms of point-counts on Bott--Samelson varieties over finite fields. More precisely, for a chosen reduced word \(R\) of \(w_0\), the probability distribution
\[
\mathbb P_{R,q}(w)
=
\frac{|\pi^{-1}_R(X^\circ_w)|}{(1+q)^{\binom n2}}
\]
is a special case of distributions obtained through random pipe dreams. Here all
varieties and maps are taken over \(\F_q\).

The Bott--Samelson perspective also suggests a natural generalization. Instead of only considering the word coming from the staircase pipe dream construction, one can attach such a probability distribution to more general words. Thus it becomes natural to ask which property of a word is remembered by this generalized probability distribution.

Motivated by the connection established above, we ask the following question:

\medskip

\begin{quest}\label{quest:main}
    
 Let \(R_1\) and \(R_2\) be reduced words in \(S_n\). When do the two distributions \(\mathbb P_{R_1,q}\) and \(\mathbb P_{R_2,q}\) agree as functions of \(q\)?

 \end{quest}

\medskip

Our main result is the following:

\begin{theorem}\label{thm:main}
Let \(R_1\) and \(R_2\) be two reduced words in \(S_n\). Then
\[
\mathbb P_{R_1,q}=\mathbb P_{R_2,q}
\]
as functions of \(q\) if and only if \(R_1\) and \(R_2\) lie in the same commutation class.
\end{theorem}

In other words, two reduced words give the same Bott--Samelson distribution precisely when they differ only by commutation moves. In particular, equality of distributions forces the two reduced words to represent the same permutation. The distribution \(\mathbb P_{R,q}\) is defined geometrically by point-counting when \(q\) is a prime power; after the combinatorial and Hecke-algebraic interpretation in \cref{subsec:hecke-characterization}, we regard it as a rational function of the indeterminate \(q\).

The rest of the paper is organized as follows. In \cref{sec:preliminary}, we review the basic objects used throughout the paper, including reduced words, the Demazure product, Schubert varieties, Bott--Samelson varieties, and the graph of commutation classes of \(w_0\). In \cref{sec:main-results}, we prove the main theorem. We first translate the Bott--Samelson point-counting distribution into a combinatorial and Hecke-algebraic form, reduce the general case to the longest element \(w_0\), and then use commutation-class invariants to characterize when two reduced words give the same distribution. In \cref{sec:further-remarks}, we record further Hecke-theoretic remarks and explain why the Temperley--Lieb quotient is too coarse for the present problem. Finally, \cref{app:hecke} recalls the Hecke-algebraic background used in the last section.

\section{Preliminaries}\label{sec:preliminary}

In this section we review the basic notation and background used throughout the paper, including reduced words, Demazure products, Schubert varieties, Bott--Samelson varieties, and commutation classes. For a comprehensive treatment of Coxeter groups and reduced words, see \cite{Humphreys1990,BjornerBrenti2005}.


\setcounter{subsection}{-1}
\subsection{Notation}\label{subsec:notation}

We fix the following notation throughout the paper.

Let \(S_n\) denote the symmetric group on \(n\) letters with standard Coxeter generators 
$s_i = (i \; i+1)$ for $1 \leq i \leq n-1$.  
The Coxeter relations are
\[
s_i^2=e,\quad 
s_i s_j=s_j s_i \ \text{for } |i-j|>1,\quad
s_i s_{i+1}s_i=s_{i+1}s_is_{i+1},
\]
where \(e\) denotes the identity element of \(S_n\).

For $w\in S_n$, let $\ell(w)$ denote its Coxeter length, and let
$\operatorname{Red}(w)$ be the set of all reduced words for $w$.

A \emph{word} is an ordered sequence
\[
R=i_1i_2\cdots i_L.
\]
Its \emph{generalized length} is defined by $L(R):=L$, and its \emph{Coxeter length} is
\[
\ell(R):=\ell\bigl(s_{i_1}s_{i_2}\cdots s_{i_L}\bigr).
\]
We call $R$ a \emph{reduced word} if $L(R)=\ell(R)$.

A word $T=t_1t_2\cdots t_m$ is called a \emph{subword} of $R=i_1i_2\cdots i_L$
if there exist indices $1\le a_1<\cdots<a_m\le L$ such that
$t_r=i_{a_r}$ for all $1\le r\le m$; equivalently, $T$ is obtained by selecting a
subsequence of digits from $R$.

Given two words
\[
R_1=i_1\cdots i_p,\qquad R_2=j_1\cdots j_q,
\]
we define their concatenation by
\[
\Cat(R_1,R_2):=i_1\cdots i_pj_1\cdots j_q.
\]
For notational simplicity, we will usually write
\[
R_1R_2
\]
in place of \(\Cat(R_1,R_2)\).

For a word \(R=i_1i_2\cdots i_L\), define
\[
w(R):=s_{i_1}s_{i_2}\cdots s_{i_L}.
\]
\[
\Dem(R):=s_{i_1}\ast s_{i_2}\ast\cdots\ast s_{i_L}.
\]
We will introduce the operator $*$ later. 

For a permutation $w\in S_n$, we use one-line notation $w=w(1)w(2)\cdots w(n)$; unless otherwise stated, this convention will be used throughout, and it should not be confused with the notation for a word $R=i_1i_2\cdots i_L$ in simple reflections.

When a word is required to be reduced, this will be stated explicitly. In particular, several constructions below, including Bott--Samelson varieties and the weighted subword model, are defined for arbitrary words.







\subsection{Reduced words and Demazure products}\label{subsec:reduced-demazure}
A reduced word representing $w\in S_n$ is a word $R=i_1\cdots i_{\ell(w)}$ such that
\[
w=s_{i_1}s_{i_2}\cdots s_{i_{\ell(w)}}.
\]

We first fix the two local moves on words that will be used throughout.
\begin{enumerate}
    \item \emph{Commutation move}: replace a consecutive subword \(ij\) by \(ji\) whenever \(|i-j|>1\).
    \item \emph{Braid move}: replace a consecutive subword $i(i+1)i$ by $(i+1)i(i+1)$ (or conversely).
\end{enumerate}

In particular, two reduced words $R=i_1\cdots i_\ell$ and $R'=j_1\cdots j_\ell$ for the same permutation $w\in S_n$ are said to differ by a \emph{single braid move} if there exists an index $t$ such that
\[
(i_t,i_{t+1},i_{t+2})=(k,k+1,k),\qquad
(j_t,j_{t+1},j_{t+2})=(k+1,k,k+1)
\]
for some $1\le k<n-1$, and $i_p=j_p$ for all $p\notin\{t,t+1,t+2\}$. That is,
\[
R=T_1\,k\,(k+1)\,k\,T_2,\qquad
R'=T_1\,(k+1)\,k\,(k+1)\,T_2,
\]
where \(T_1\) and \(T_2\) are two words. The definition of differing by a single commutation move is analogous. We write \(R\sim_C R'\) if \(R\) and \(R'\) differ by a single commutation move, and \(R\sim_B R'\) if they differ by a single braid move.

We will use the following fundamental result \cite[Theorem~3.3.1]{BjornerBrenti2005}.

\begin{theorem}[Matsumoto]\label{thm:Matsumoto}
For any $w\in S_n$ and any $R_1,R_2\in \Red(w)$, one can pass from $R_1$ to $R_2$ by a finite sequence of commutation moves and braid moves.
\end{theorem}

\begin{ex}
Let $w_0=54321\in S_5$ be the longest element.
Consider the two reduced words
\[
R_1=1213214321,
\qquad
R_2=1232143231.
\]
They are connected by a sequence involving both commutation moves  and braid moves:
\[
\begin{aligned}
1213214321
& \sim_C 1231214321
\sim_B 1232124321 \\
&\sim_C 1232142321
\sim_B 1232143231.
\end{aligned}
\]

\end{ex}

We now introduce the \emph{Demazure product} on $S_n$ and list the basic properties that will be used later.

\begin{defin}\label{def:demazure-simple}
For \(w\in S_n\) and \(1\le i\le n-1\), define
\[
w*s_i=
\begin{cases}
ws_i, & \ell(ws_i)>\ell(w),\\[4pt]
w, & \ell(ws_i)<\ell(w).
\end{cases}
\]
This operation is the elementary Demazure product with a simple reflection.
\end{defin}

The first crucial property of the Demazure product is that it is compatible with the Coxeter relations.

\begin{proposition}\label{prop:demazure-relations}
For every \(w\in S_n\), the elementary Demazure product satisfies the Coxeter relations:
\begin{enumerate}
    \item If \(|i-j|>1\), then
    \[
    w*s_i*s_j=w*s_j*s_i.
    \]
    \item For \(1\le i\le n-2\), one has
    \[
    w*s_i*s_{i+1}*s_i=w*s_{i+1}*s_i*s_{i+1}.
    \]
\end{enumerate}
Consequently, the product \(w*s_{i_1}*\cdots *s_{i_\ell}\) depends only on the element represented by the reduced word \(s_{i_1}\cdots s_{i_\ell}\).
\end{proposition}

\begin{proof}
This is the standard 0-Hecke, or Demazure, form of the Coxeter braid relations; see \cite[Theorem~7.1]{Humphreys1990}.
\end{proof}

Hence the simple Demazure product extends unambiguously to a binary operation on $S_n$.

\begin{defin}\label{def:demazure-permutations}
Let \(\alpha,\beta\in S_n\). Choose a reduced expression
\[
\beta=s_{i_1}s_{i_2}\cdots s_{i_\ell}.
\]
The \emph{Demazure product} of \(\alpha\) and \(\beta\) is
\[
\alpha*\beta
=
\alpha*s_{i_1}*s_{i_2}*\cdots *s_{i_\ell}.
\]
By \cref{prop:demazure-relations}, this definition is independent of the chosen reduced expression for \(\beta\).
\end{defin}

It is often convenient to define the Demazure product directly for words.

\begin{defin}\label{def:demazure-words}
Let \(R=i_1i_2\cdots i_L\) be a word in the alphabet \(\{1,\ldots,n-1\}\). Its \emph{word Demazure product} is
\[
\Dem(R):=e*s_{i_1}*s_{i_2}*\cdots *s_{i_L}.
\]
More generally, for two words \(R_1,R_2\), we set
\[
\Dem(R_1R_2):=\Dem(\Cat(R_1,R_2)).
\]
\end{defin}

\begin{remark}\label{rem:word-demazure-dependence}
The word Demazure product depends on the word, not only on its ordinary product in \(S_n\). For example, \(\Dem(ii)=s_i\), whereas \(w(ii)=e\). However, if \(R_1\) and \(R_2\) are reduced expressions for the same permutation, then \(\Dem(R_1)=\Dem(R_2)\) by \cref{prop:demazure-relations}.
\end{remark}


\subsection{Flag varieties and Schubert varieties}\label{subsec:flag-varieties}

Let $\mathbb K$ be a field, let $G=\GL_n(\mathbb K)$, and let
$B\subset G$ be the Borel subgroup of invertible upper triangular matrices.
The complete flag variety is
\[
\mathcal F\ell_n=G/B
=
\left\{
F_\bullet:\ 0=F_0\subset F_1\subset\cdots\subset F_n=\mathbb K^n,
\ \dim F_i=i
\right\}.
\]
Let $E_\bullet$ denote the standard flag, where
\[
E_i=\langle e_1,e_2,\ldots,e_i\rangle.
\]

For $w\in S_n$, define the coordinate flag $E_\bullet^w$ by
\[
E_i^w=\langle e_{w(1)},e_{w(2)},\ldots,e_{w(i)}\rangle,
\qquad 0\le i\le n.
\]
The Schubert cell indexed by $w$ is
\[
X_w^\circ:=B\cdot E_\bullet^w\subset \mathcal F\ell_n.
\]
Equivalently, the complete flag variety admits the Bruhat decomposition
\[
\mathcal F\ell_n=\bigsqcup_{w\in S_n}X_w^\circ.
\]
The Schubert variety associated to $w$ is the closure
\[
X_w:=\overline{X_w^\circ}.
\]
Throughout this subsection, the base field is understood from context. Thus
\(X_w^\circ\) and \(X_w\) denote the Schubert cell and Schubert variety over the
current base field.

\subsection{Bott--Samelson varieties}\label{subsec:bott-samelson}

Let $G=\GL_n(\mathbb K)$, let $B\subset G$ be the Borel subgroup of invertible
upper triangular matrices, and let $E_\bullet$ be the standard flag in
$\mathbb K^n$.

\begin{defin}\label{def:bott-samelson-flag}
Let $R=i_1i_2\cdots i_L$ be a word in $\{1,\ldots,n-1\}$.
The \emph{Bott--Samelson variety} associated to $R$ is the variety
\[
\mathrm{BS}^R
=
\left\{
(F_\bullet^{(0)},F_\bullet^{(1)},\ldots,F_\bullet^{(L)})
\ \middle|\
\begin{array}{l}
F_\bullet^{(0)}=E_\bullet,\\[2pt]
F_\bullet^{(r)}\in G/B \text{ for }0\le r\le L,\\[2pt]
F_k^{(r-1)}=F_k^{(r)}
\text{ for all }k\neq i_r,\ 1\le r\le L
\end{array}
\right\}.
\]
In other words, at the $r$-th step, the two consecutive flags
$F_\bullet^{(r-1)}$ and $F_\bullet^{(r)}$ agree in every dimension except
possibly dimension $i_r$.
\end{defin}

The Bott--Samelson map is the projection to the last flag:
\begin{equation}\label{def:pi}
\pi_R:\mathrm{BS}^R\longrightarrow G/B,
\qquad
(F_\bullet^{(0)},F_\bullet^{(1)},\ldots,F_\bullet^{(L)})
\longmapsto F_\bullet^{(L)}.
\end{equation}

When $R$ is a reduced word for $w\in S_n$, the image of $\pi_R$ is the Schubert
variety $X_w\subset G/B$. In this case, $\mathrm{BS}^R$ is smooth of dimension
$\ell(w)$, and $\pi_R$ is a resolution of singularities of $X_w$
\cite{BottSamelson1958,Demazure1974,Hansen1973}.

The notation \(\mathrm{BS}^R\) always refers to the Bott--Samelson variety
over the current base field. For a locally closed subvariety \(Y\subset G/B\),
the symbol \(\pi_R^{-1}(Y)\) denotes its inverse image under \(\pi_R\).

Over a finite field $\mathbb K=\mathbb F_q$, each step in the definition of
$\mathrm{BS}^R$ has $q+1$ choices. Indeed, after fixing
\[
F_{i_r-1}^{(r)}=F_{i_r-1}^{(r-1)}
\qquad\text{and}\qquad
F_{i_r+1}^{(r)}=F_{i_r+1}^{(r-1)},
\]
the possible choices of $F_{i_r}^{(r)}$ are precisely the intermediate subspaces
\[
F_{i_r-1}^{(r-1)}
\subset F_{i_r}^{(r)}
\subset F_{i_r+1}^{(r-1)}.
\]
Equivalently, they are the lines in the two-dimensional quotient
\[
F_{i_r+1}^{(r-1)}/F_{i_r-1}^{(r-1)}.
\]
Thus the choices form a projective line $\mathbb P^1$, which has
$q+1$ points. Therefore
\[
\left|\mathrm{BS}^R\right|=(q+1)^{L(R)}.
\]

The map $\pi_R$ induces a probability distribution on $S_n$ by
\begin{equation}\label{eq:BS-probability}
\mathbb P_{R,q}(w)
=
\frac{
\left|
\pi_R^{-1}(X_w^\circ)
\right|
}{
(q+1)^{L(R)}
}.
\end{equation}

\begin{ex}\label{ex:BS-GL3}
Let $G=\GL_3$ and let $R=121$. Then $\mathrm{BS}^R$ consists of sequences of
complete flags
\[
(F_\bullet^{(0)},F_\bullet^{(1)},F_\bullet^{(2)},F_\bullet^{(3)})
\]
such that $F_\bullet^{(0)}=E_\bullet$ and
\[
F_k^{(0)}=F_k^{(1)} \quad (k\neq 1),
\]
\[
F_k^{(1)}=F_k^{(2)} \quad (k\neq 2),
\]
\[
F_k^{(2)}=F_k^{(3)} \quad (k\neq 1).
\]
Equivalently, only the $1$-dimensional subspace may change in the first step,
only the $2$-dimensional subspace may change in the second step, and only the
$1$-dimensional subspace may change in the third step.
\end{ex}
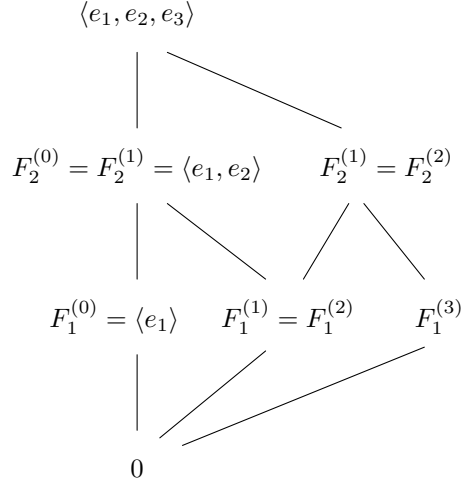
\begin{figure}[htbp]
\centering
\begin{tikzpicture}[scale=1.0, transform shape, line cap=round, line join=round]
    \node at (0,0) {$0$};

    \draw (0,0.5) -- (0,1.5);
    \node at (-0.3,2) {$F_1^{(0)} = \langle e_1 \rangle$};

    \draw (0,2.5) -- (0,3.5);
    \node at (0,4) {$F_2^{(0)} = F_2^{(1)} = \langle e_1,e_2 \rangle$};

    \draw (0,4.5) -- (0,5.5);
    \node at (0,6) {$\langle e_1,e_2,e_3 \rangle$};

    \draw (0.3,0.4) -- (1.7,1.55);
    \node at (2,2) {$F_1^{(1)} = F_1^{(2)}$};

    \draw (1.7,2.5) -- (0.4,3.5);
    \node at (3.3,4) {$F_2^{(1)} = F_2^{(2)}$};

    \draw (2.2,2.5) -- (2.8,3.5);
    \draw (2.7,4.5) -- (0.4,5.5);

    \node at (4,2) {$F_1^{(3)}$};
    \draw (0.6,0.3) -- (3.8,1.6);
    \draw (3.8,2.5) -- (3,3.5);
\end{tikzpicture}
\caption{The Bott--Samelson variety \(\mathrm{BS}^{121}\).}
\label{fig:BS121}
\end{figure}

\subsection{Graph of commutation classes \texorpdfstring{\(\widehat G(w_0)\)}{Ghat(w0)}}\label{subsec:commutation-classes}

We briefly recall the graph-theoretic structure on reduced words of the longest element.
The discussion in this subsection is adapted from \cite{GutierresMamedeSantos2020}.

We recall the relations:
\[
s_is_j=s_js_i \qquad (|i-j|>1)
\]
\emph{commutation moves}, and the relations
\[
s_is_{i+1}s_i=s_{i+1}s_is_{i+1}
\]
\emph{braid moves}.

Fix a permutation $w \in S_n$, we now define an equivalence relation on $\operatorname{Red}(w)$. For $R,R'\in \operatorname{Red}(w)$, write
\[
R\sim R'
\]
if $R$ can be obtained from $R'$ by a finite sequence of commutation moves.
This is an equivalence relation, and its equivalence classes are called the
\emph{commutation classes} of $w$.
We denote the set of commutation classes by
\[
\mathcal C(w),
\]
and for $R\in \operatorname{Red}(w)$ we write $[R]$ for the commutation class containing $R$.

For commutation classes \([R],[R']\in\mathcal C(w)\), we write
 $[R]\text{---}[R']$
if there exist representatives $\widetilde{R}\in [R]$ and $\widetilde{R}'\in [R']$ such that
\[
\widetilde{R}\sim_B \widetilde{R}'.
\]

\begin{lemma}[\cite{GutierresMamedeSantos2020}, Lemma 1]
\label{lem:commutation-characterization}
Let $R$ and $R'$ be two words over the alphabet $[n-1]$. Then
\[
R \sim R'
\quad\Longleftrightarrow\quad
R|_{\{i,i+1\}} = R'|_{\{i,i+1\}}
\text{ for all } i \in [n-2].
\]
\end{lemma}

Let $w_0\in S_n$ be the longest element. In one-line notation,
\[
w_0=n(n-1)\cdots 321,
\]
and
\[
\ell(w_0)=\binom{n}{2}.
\]
A standard reduced word for $w_0$ is
\[
R_0=1(21)(321)\cdots ((n-1)\cdots 21).
\]

We can now define the graph of commutation classes.
Its vertex set is \(\mathcal C(w_0)\), and two vertices $[R]$ and $[R']$ are joined by an edge
whenever
\[
[R]\text{---} [R'].
\]
Following \cite{GutierresMamedeSantos2020}, we denote this graph by $\widehat{G}(w_0)$.

This graph may be viewed as a quotient of a larger graph $G(w_0)$. 
Namely, $G(w_0)$ is the graph whose vertex set is $\operatorname{Red}(w_0)$, and where two reduced
words are adjacent if they differ by a single commutation move or a single braid move.
Thus $\widehat{G}(w_0)$ is obtained from $G(w_0)$ by contracting each commutation class
to a single vertex. In particular, by \cref{thm:Matsumoto} we can see that $G(w_0)$ is connected.

We now recall some invariants which are crucial for the structure of the graph. The first one is the sum of digits of a word:

\begin{defin}\label{def:digit-sum}
For a reduced word
\[
R=i_1i_2\cdots i_{\ell}\in \Red(w_0),
\]
define its \emph{digit sum} by
\[
S(R):=\sum_{j=1}^{\ell} i_j.
\]

\end{defin}

\begin{defin}
For \(R=i_1i_2\cdots i_{\ell}\in \Red(w_0)\), define its \emph{complement} and \emph{reverse} by
\[
R^{\bullet}:=(n-i_1)(n-i_2)\cdots (n-i_{\ell}),
\qquad
R^{\mathrm{Rev}}:=i_{\ell}\cdots i_2i_1.
\]
These are involutions on \(\Red(w_0)\), and it is easy to check that they commute with each other:
\[
(R^{\bullet})^{\mathrm{Rev}}=(R^{\mathrm{Rev}})^{\bullet}.
\]
\end{defin}

\begin{lemma}
The digit sum \(S(R)\) is constant on each commutation class of \(w_0\). Hence \(S\) defines a map on the vertex set of \(\widehat{G}(w_0)\).
\end{lemma}

\begin{proof}
A commutation move only interchanges two adjacent letters \(i\) and \(j\), so it does not change the sum of the letters. Therefore \(S(R)\) depends only on the commutation class of \(R\).
\end{proof}

\begin{lemma}\label{lem:digit-sum-edge}
If two reduced words \(R,R'\in \Red(w_0)\) differ by a single braid move, then
\[
|S(R)-S(R')|=1.
\]
Consequently, if two commutation classes are joined by an edge in \(\widehat{G}(w_0)\), then the absolute difference of their digit sums is \(1\).
\end{lemma}

\begin{proof}
A braid move replaces a factor \(i(i+1)i\) by \((i+1)i(i+1)\), or vice versa. The corresponding digit sums differ by
\[
\bigl((i+1)+i+(i+1)\bigr)-\bigl(i+(i+1)+i\bigr)=1.
\]
Hence the absolute difference is \(1\).
\end{proof}

The second one plays a key role in our following proof of the main theorem, which also gives a characterization of the commutation classes.


\begin{defin}
Let
\[
\mathcal{T}_n:=\{(a,b,c)\in [n]^3:\ a<b<c\}.
\]
For a reduced word \(R\in \operatorname{Red}(w_0)\), say $R=i_1i_2\cdots i_N$, where $N=\binom{n}{2}$. We see $R$ as an operation sequence: starting from the identity permutation, apply $s_{i_1},\dots,s_{i_N}$ successively.
For $a<b$, let $\tau_R(a,b)\in\{1,\dots,N\}$ be the unique time at which the pair $(a,b)$ becomes an inversion for the first time.
\end{defin}

Note that for \(R\in \operatorname{Red}(w_0)\), every pair $a<b$ inverts exactly once, so $\tau_R$ is well-defined.

\begin{defin}
    For $(a,b,c)\in \mathcal{T}_n$, define
    \[
    T_R(a,b,c)=
    \begin{cases}
     1, & \text{if }\tau_R(a,b)<\tau_R(b,c),\\
    -1, & \text{otherwise.}
    \end{cases}
    \]   
\end{defin}

This invariant (in equivalent guises) appears in 
\cite[\textsection2, Lemma~10]{GutierresMamedeSantos2020}.
We will mainly use two of its properties presented in \cite{GutierresMamedeSantos2020}.
For the first property, we give a self-contained proof based on the following linear-extension lemma, in a form tailored to the notation used here.

\begin{lemma}[Linear-extension lemma {\cite[Lemma~1]{Eti84}}]\label{lem:linear-extension}
Let $(X,\prec)$ be a finite poset. A \emph{linear extension} of $(X,\prec)$ is a total order (equivalently, a listing
$x_1,\dots,x_{|X|}$ of the elements of $X$) such that $x\prec y$ implies $x$ appears before $y$.
Then any two linear extensions of $(X,\prec)$ are connected by a sequence of swaps of adjacent
\emph{incomparable} elements.
\end{lemma}

Here we state the lemma without proof. One can see \cite{Eti84} for the proof.

\begin{proposition}[{\cite[Lemma~10]{GutierresMamedeSantos2020}}]\label{prop:commutation-classes-via-T}
   For $R_1,R_2\in \operatorname{Red}(w_0)$, one has $[R_1]=[R_2]$ if and only if
    \[
    T_{R_1}(a,b,c)=T_{R_2}(a,b,c)\quad\text{for all }(a,b,c)\in \mathcal{T}_n.
    \] 
\end{proposition}

\begin{proof}

For the ``only if'' direction, observe that commutation moves do not change $T_R(a,b,c)$ for any triple, thus $T_R$ depends only on $[R]$, and we may write $T_{[R]}(a,b,c)$ unambiguously.

For the converse, assume $T_{R_1}(a,b,c)=T_{R_2}(a,b,c)$ for all $(a,b,c)\in \mathcal{T}_n$, and write the common value as $T(a,b,c)$. Let
\[
P:=\{(a,b):1\le a<b\le n\}.
\]
Define a directed relation $R_T$ on $P$ by imposing, for every $a<b<c$,
\[
\begin{cases}
(a,b)\to(a,c)\to(b,c), & \text{if }T(a,b,c)=1,\\
(b,c)\to(a,c)\to(a,b), & \text{if }T(a,b,c)=-1.
\end{cases}
\]
Let $\prec_T$ be the transitive closure of $R_T$.

\smallskip
\noindent\emph{Claim 1:} $(P,\prec_T)$ is a strict poset.
Fix any \(R\in \Red(w_0)\) with $T_R\equiv T$ (e.g.\ $R=R_1$). For each $a<b<c$, the restriction to the three values
$a,b,c$ yields a reduced expression for the longest element in $S_3$, so necessarily $\tau_R(a,c)$ lies strictly between
$\tau_R(a,b)$ and $\tau_R(b,c)$. Hence every edge $p\to q$ in $R_T$ satisfies $\tau_R(p)<\tau_R(q)$, so $R_T$ has no directed cycle.
Therefore its transitive closure $\prec_T$ is transitive and irreflexive.

\smallskip
\noindent\emph{Claim 2:} The event sequences $E(R_1)$ and $E(R_2)$ are linear extensions of $(P,\prec_T)$.
Define $E(R)$ to be the list of all pairs in $P$ in increasing order of $\tau_R$.
Since each edge in $R_T$ respects $\tau_{R_1}$ (and $\tau_{R_2}$), both $E(R_1)$ and $E(R_2)$ respect $\prec_T$.
By lemma~\ref{lem:linear-extension},
there exists a sequence of adjacent swaps of \emph{incomparable} elements transforming $E(R_1)$ into $E(R_2)$.

\smallskip
\noindent\emph{Claim 3:} If $p=(a,b)$ and $q=(c,d)$ are incomparable in $(P,\prec_T)$, then $\{a,b\}\cap\{c,d\}=\varnothing$.
Indeed, if $p$ and $q$ share an entry, say $p=(a,b)$ and $q=(a,c)$ with $b<c$, then the triple $a<b<c$ forces
$(a,b)\prec_T(a,c)$ or $(a,c)\prec_T(a,b)$ by construction, contradicting incomparability. The other shared-entry cases are analogous.

\smallskip
Now interpret one adjacent swap in the event sequence. At time $t$, the generator $s_{i_t}$ swaps the two adjacent entries
currently occupying positions $i_t,i_t+1$; denote those values by $u_t,v_t$. The unique new inversion created at time $t$ is precisely
the pair $\left(\min\{u_t,v_t\},\max\{u_t,v_t\}\right)$, i.e.\ the $t$-th element of $E(R)$.
If two consecutive events are disjoint pairs of values, then the corresponding swaps occur on disjoint adjacent blocks, hence the two
simple reflections commute (their indices differ by $>1$). By Claim~3, every adjacent swap used to pass from $E(R_1)$ to $E(R_2)$ is of this form.
Therefore the entire swap sequence lifts to a sequence of commutation moves from $R_1$ to $R_2$, i.e.\ $[R_1]=[R_2]$.
\end{proof}

\begin{proposition}[{\cite[Lemma~11]{GutierresMamedeSantos2020}}]\label{prop:braid-flips-one-T}
Let $R_1,R_2\in \operatorname{Red}(w_0)$ differ by a single braid move
\[
\cdots\, i\, (i+1)\, i\, \cdots \quad\longleftrightarrow\quad \cdots\, (i+1)\, i\, (i+1)\, \cdots .
\]
Then $T_{R_1}(a,b,c)=T_{R_2}(a,b,c)$ for all triples $(a,b,c)\in \mathcal{T}_n$, except for one.
\end{proposition}

A proof can be found in \cite{GutierresMamedeSantos2020}, hence we omit it.

\begin{ex}\label{ex:braid-one-triple}
Let \(n=4\), so \(w_0=4321\), and
\[
\mathcal{T}_n=\{(1,2,3),(1,2,4),(1,3,4),(2,3,4)\}.
\]
Consider the two reduced words
\[
R=121321,\qquad R'=212321.
\]
They differ by the braid move \(121\leftrightarrow 212\) at the beginning.

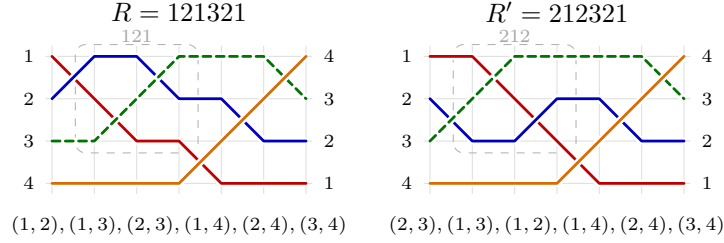
\begin{center}
\begin{tikzpicture}[
    x=0.56cm,
    y=0.56cm,
    wire/.style={
        line width=1.05pt,
        line cap=round,
        line join=round,
        preaction={draw=white,line width=2.8pt}
    },
    gridline/.style={gray!25, thin},
    braidbox/.style={gray!60, dashed, rounded corners, thin}
]

\begin{scope}
\node at (3,4.05) {\(R=121321\)};

\foreach \x in {0,...,6} {
    \draw[gridline] (\x,-0.25) -- (\x,3.25);
}
\foreach \y in {0,...,3} {
    \draw[gridline] (-0.15,\y) -- (6.15,\y);
}

\draw[braidbox] (0.55,0.72) rectangle (3.45,3.28);
\node[gray!70] at (2,3.48) {\scriptsize \(121\)};

\foreach \lab/\y in {1/3,2/2,3/1,4/0} {
    \node[left] at (-0.2,\y) {\scriptsize \(\lab\)};
}
\foreach \lab/\y in {4/3,3/2,2/1,1/0} {
    \node[right] at (6.2,\y) {\scriptsize \(\lab\)};
}

\draw[wire, red!75!black]
    (0,3) -- (1,2) -- (2,1) -- (3,1) -- (4,0) -- (5,0) -- (6,0);
\draw[wire, blue!75!black]
    (0,2) -- (1,3) -- (2,3) -- (3,2) -- (4,2) -- (5,1) -- (6,1);
\draw[wire, green!45!black, densely dashed]
    (0,1) -- (1,1) -- (2,2) -- (3,3) -- (4,3) -- (5,3) -- (6,2);
\draw[wire, orange!85!black]
    (0,0) -- (1,0) -- (2,0) -- (3,0) -- (4,1) -- (5,2) -- (6,3);

\node[below] at (3,-0.55)
{\scriptsize \((1,2),(1,3),(2,3),(1,4),(2,4),(3,4)\)};
\end{scope}

\begin{scope}[xshift=5cm]
\node at (3,4.05) {\(R'=212321\)};

\foreach \x in {0,...,6} {
    \draw[gridline] (\x,-0.25) -- (\x,3.25);
}
\foreach \y in {0,...,3} {
    \draw[gridline] (-0.15,\y) -- (6.15,\y);
}

\draw[braidbox] (0.55,0.72) rectangle (3.45,3.28);
\node[gray!70] at (2,3.48) {\scriptsize \(212\)};

\foreach \lab/\y in {1/3,2/2,3/1,4/0} {
    \node[left] at (-0.2,\y) {\scriptsize \(\lab\)};
}
\foreach \lab/\y in {4/3,3/2,2/1,1/0} {
    \node[right] at (6.2,\y) {\scriptsize \(\lab\)};
}

\draw[wire, red!75!black]
    (0,3) -- (1,3) -- (2,2) -- (3,1) -- (4,0) -- (5,0) -- (6,0);
\draw[wire, blue!75!black]
    (0,2) -- (1,1) -- (2,1) -- (3,2) -- (4,2) -- (5,1) -- (6,1);
\draw[wire, green!45!black, densely dashed]
    (0,1) -- (1,2) -- (2,3) -- (3,3) -- (4,3) -- (5,3) -- (6,2);
\draw[wire, orange!85!black]
    (0,0) -- (1,0) -- (2,0) -- (3,0) -- (4,1) -- (5,2) -- (6,3);

\node[below] at (3,-0.55)
{\scriptsize \((2,3),(1,3),(1,2),(1,4),(2,4),(3,4)\)};
\end{scope}

\end{tikzpicture}

\medskip
{\small
The wiring diagrams for \(R=121321\) and \(R'=212321\).
The dashed boxes mark the braid move \(121\leftrightarrow 212\).
The sequence below each diagram records the newly created value-inversion pairs.
}
\end{center}

For \(R\), the sequence of newly created value-inversion pairs is
\[
(1,2),\ (1,3),\ (2,3),\ (1,4),\ (2,4),\ (3,4),
\]
whereas for \(R'\), it is
\[
(2,3),\ (1,3),\ (1,2),\ (1,4),\ (2,4),\ (3,4).
\]
Hence the corresponding \(\tau\)-values are
\[
\begin{array}{c|cccccc}
        & (1,2) & (1,3) & (2,3) & (1,4) & (2,4) & (3,4)\\
\hline
\tau_R  & 1 & 2 & 3 & 4 & 5 & 6\\
\tau_{R'} & 3 & 2 & 1 & 4 & 5 & 6
\end{array}
\]

We now evaluate \(T_R(a,b,c)\) and \(T_{R'}(a,b,c)\) for all triples in \(\mathcal{T}_n\):
\[
\begin{array}{c|ccc}
(a,b,c) & T_R(a,b,c) & T_{R'}(a,b,c) & \text{different?}\\
\hline
(1,2,3) & +1 & -1 & \ast\\
(1,2,4) & +1 & +1 & \\
(1,3,4) & +1 & +1 & \\
(2,3,4) & +1 & +1 &
\end{array}
\]
Indeed, \(R\) and \(R'\) differ by exactly one braid move \(121\leftrightarrow 212\), and
\[
T_R(a,b,c)\neq T_{R'}(a,b,c)
\]
only for the triple \((1,2,3)\), namely the unique triple supported on the moved block
\(\{1,2,3\}\). This illustrates Proposition~\ref{prop:commutation-classes-via-T}
and Proposition~\ref{prop:braid-flips-one-T}.
\end{ex}

\begin{cor}\label{cor:hypercube-embedding}
The assignment
\[
[R]\longmapsto \bigl(T_R(a,b,c))\bigr)_{(a,b,c)\in \mathcal{T}_n}
\]
is well-defined and injective on the set of commutation classes of \(w_0\). In particular, after identifying \(\{-1,1\}\) with \(\mathbb Z_2\), the graph \(\widehat{G}(w_0)\) embeds into the hypercube
$\mathbb Z_2^{\binom{n}{3}}.$

\end{cor}

\begin{remark}
 In general, this embedding is not an isometry. See \cite{Elnitsky1997},\cite{AwikBrelandCadmanErnst2024}.
\end{remark}

\section{Proof of the main results}\label{sec:main-results}

In this section, we prove the main theorem. For an arbitrary word
\(R=i_1\cdots i_L\), let \(\mathrm{BS}^R\) be the Bott--Samelson variety over
\(\mathbb F_q\), and let \(\pi_R\) be the Bott--Samelson map defined in
\eqref{def:pi}. When \(R\in\Red(w_0)\), we write
\(
N=\binom{n}{2}.
\)

We first translate the Bott--Samelson point-counting distribution into a weighted subword model and then into an identity in the Hecke algebra. This gives an algebraic criterion for two words to define the same distribution. We then reduce the general case to reduced words for the longest element \(w_0\). In that case, the coefficient of degree \(N-1\) in the Hecke expansion records one-letter Demazure shortenings. The key point is to organize these shortenings into a signed interval statistic \(\Theta_R\), whose variation under braid moves is exactly opposite to the variation of the commutation-class invariant \(T_R\). This allows us to recover \(T_R\) from the distribution and hence distinguish commutation classes.

\subsection{A Hecke-algebraic characterization}\label{subsec:hecke-characterization}

\begin{defin}\label{def:BS-distribution}
Let \(R=i_1\cdots i_L\) be a word in the alphabet \(\{1,\ldots,n-1\}\). For a prime power \(q\), define
\(
\mathbb P_{R,q}:S_n\to[0,1]
\)
by
\begin{equation}\label{eqn:prob}
\mathbb P_{R,q}(w)
=
\frac{
\left|\pi_R^{-1}(X_w^\circ)\right|
}{
(q+1)^{L(R)}
}.
\end{equation}
\end{defin}

\begin{lemma}
    The function $\mathbb{P}_{R,q}$ defines a probability distribution on $S_n$.
\end{lemma}

\begin{proof}
By the Bruhat decomposition,
\[
\mathcal F\ell_n=\bigsqcup_{w\in S_n}X_w^\circ.
\]
Therefore the sets \(\pi_R^{-1}(X_w^\circ)\), for \(w\in S_n\), form a disjoint decomposition of \(\mathrm{BS}^R\). Hence
\[
\sum_{w\in S_n}
\left|\pi_R^{-1}(X_w^\circ)\right|
=
\left|\mathrm{BS}^R\right|
=
(q+1)^{L(R)}.
\]
It follows that \(\sum_{w\in S_n}\mathbb P_{R,q}(w)=1\). Nonnegativity is immediate from the definition.
\end{proof}

For two words \(R\) and \(R'\), we write
\(
R\equiv R'
\)
if
\[
\mathbb P_{R,q}(w)=\mathbb P_{R',q}(w)
\qquad\text{for all }w\in S_n
\]
as rational functions of \(q\).

\begin{defin}\label{def:comb}
Let \(R=i_1i_2\cdots i_L\) be a word in the alphabet \(\{1,\ldots,n-1\}\).
By definition, a subword \(T\) of \(R\) is specified by a strictly increasing sequence of positions
\[
a(T)=(a_1<a_2<\cdots<a_m),\qquad a_r\in [L].
\]
The associated word is
\[
T=i_{a_1}i_{a_2}\cdots i_{a_m}.
\]

We define the weight \(\operatorname{wt}_R(T)\) as follows. Start with
\[
x_0=e,\qquad d_0=0.
\]
For \(k=1,\ldots,L\), suppose \(x_{k-1}\in S_n\) and \(d_{k-1}\in \mathbb{Z}_{\ge 0}\)
have been defined. Let
\[
    \varepsilon_k(T)=
    \begin{cases}
    1, & k\in \{a_1,\ldots,a_m\},\\
    0, & k\notin \{a_1,\ldots,a_m\},
    \end{cases}
\]
and let
\[
    \gamma_k(T)=
    \begin{cases}
    1, & \ell(x_{k-1}s_{i_k})>\ell(x_{k-1}),\\
    0, & \ell(x_{k-1}s_{i_k})<\ell(x_{k-1}).
    \end{cases}
\]
We say that the \(k\)-th decision is good if
\[
    \varepsilon_k(T)=\gamma_k(T).
\]

Define
\[
    d_k=
    \begin{cases}
    d_{k-1}+1, & \text{if the \(k\)-th decision is good},\\
    d_{k-1}, & \text{otherwise}.
    \end{cases}
\]
The permutation is updated by
\[
    x_k=
    \begin{cases}
    x_{k-1}s_{i_k},
    & \text{if } k\in a(T)
      ,\\
    x_{k-1},
    & \text{otherwise}.
    \end{cases}
\]
Finally, define
\[
\operatorname{wt}_R(T):=q^{d_L}.
\]
For \(x\in S_n\), set
\[
\operatorname{wt}_R(x)
:=
\sum_{\substack{
T \text{ subword of } R\\
w(T)=x
}}
\operatorname{wt}_R(T).
\]
\end{defin}

The two indicators have the following interpretation. The value
\(\varepsilon_k(T)\) records whether the \(k\)-th letter of \(R\) is selected
in the  subword \(T\), while \(\gamma_k(T)\) records whether right
multiplication by \(s_{i_k}\) would increase the Coxeter length at the current
state \(x_{k-1}\). Thus the equality
\[
    \varepsilon_k(T)=\gamma_k(T)
\]
means that the \(k\)-th decision is compatible with the length direction:
we select the letter exactly when it increases length, and we skip it exactly
when it would decrease length. Such a decision is regarded as a good decision,
and each good decision is rewarded by a factor \(q\).

\begin{proposition}\label{prop:probdist}
Let \(R=i_1i_2\cdots i_L\) be a fixed word. Then for every \(w\in S_n\),
\[
    \mathbb{P}_{R,q}(w)
    =
    \frac{\operatorname{wt}_R(w)}{(1+q)^L}.
\]

\end{proposition}

\begin{proof}
For a word \(R\), write
\[
N_R(x):=\left|\pi_R^{-1}(X_x^\circ)\right|.
\]
It suffices to prove
\[
N_R(x)=\wt_R(x)
\]
for every \(x\in S_n\), since
\[
|\mathrm{BS}^R|=(q+1)^{L(R)}.
\]

We first recall the following rank-one fact. Fix \(i\in\{1,\ldots,n-1\}\) and
let
\[
\operatorname{pr}_i:G/B\longrightarrow G/P_i
\]
be the projection forgetting the \(i\)-dimensional subspace. If
\(F_\bullet\in X_y^\circ\), then
\[
\operatorname{pr}_i^{-1}(\operatorname{pr}_i(F_\bullet))
\cong \mathbb P^1.
\]
Moreover, this \(\mathbb P^1\) meets only the two Schubert cells
\(X_y^\circ\) and \(X_{ys_i}^\circ\). More precisely, if
\(\ell(ys_i)>\ell(y)\), then
\[
\left|
\operatorname{pr}_i^{-1}(\operatorname{pr}_i(F_\bullet))
\cap X_y^\circ
\right|=1,
\qquad
\left|
\operatorname{pr}_i^{-1}(\operatorname{pr}_i(F_\bullet))
\cap X_{ys_i}^\circ
\right|=q.
\]
If \(\ell(ys_i)<\ell(y)\), then the two numbers are reversed:
\[
\left|
\operatorname{pr}_i^{-1}(\operatorname{pr}_i(F_\bullet))
\cap X_y^\circ
\right|=q,
\qquad
\left|
\operatorname{pr}_i^{-1}(\operatorname{pr}_i(F_\bullet))
\cap X_{ys_i}^\circ
\right|=1.
\]
This follows from the Bruhat decomposition of the rank-one fiber
\(\mathbb P^1\): it consists of one closed point and one open affine cell
\(\mathbb A^1\), which has \(q\) points over \(\mathbb F_q\).

We now prove the statement by induction on \(L=L(R)\).

If \(L=0\), then \(\mathrm{BS}^{\varnothing}\) consists of the single standard flag
\(E_\bullet\), and \(\pi_{\varnothing}\) maps it to \(E_\bullet\in X_{e}^\circ\).
Hence
\[
N_{\varnothing}(e)=1,\qquad
N_{\varnothing}(x)=0\quad \text{for }x\neq e.
\]
This agrees with the definition of \(\operatorname{wt}_{\varnothing}\).

If \(L=1\), say \(R=i_1\), then \(\mathrm{BS}^R\cong \mathbb P^1\). The
Bott--Samelson map sends one point to \(X_{e}^\circ\) and the remaining
\(q\) points to \(X_{s_{i_1}}^\circ\). Hence
\[
N_R(e)=1,\qquad N_R(s_{i_1})=q,
\]
and \(N_R(x)=0\) for all other \(x\in S_n\). This agrees with \cref{def:comb}, because the empty subword has weight \(1\), while the subword consisting of the single letter \(i_1\) has weight \(q\). Thus \(N_R(x)=\wt_R(x)\) in the base case.

Now assume the result for all words of length \(L-1\). Write
\[
R=R'i,
\]
where \(i=i_L\). Fix \(x\in S_n\). A point of \(\mathrm{BS}^R\) is
obtained from a point of \(\mathrm{BS}^{R'}\), together with one choice
in the rank-one fiber of \(\operatorname{pr}_i\). By the rank-one fact above, the
last step can only move the Schubert-cell label from \(y\) to either \(y\) or
\(ys_i\). Therefore, in order to end in \(X_x^\circ\), the previous Schubert-cell
label must be either \(x\) or \(xs_i\).

There are two cases.

First suppose \(\ell(xs_i)>\ell(x)\). Then from a point ending in \(X_x^\circ\)
at step \(L-1\), exactly one choice in the last projective-line fiber remains
in \(X_x^\circ\). Also, from a point ending in \(X_{xs_i}^\circ\), exactly one
choice moves to \(X_x^\circ\). Hence
\[
N_R(x)=N_{R'}(x)+N_{R'}(xs_i).
\]
On the other hand, \cref{def:comb} gives exactly the same recurrence for the weights:
if the current state is \(x\), then skipping \(i\) is a bad decision and contributes
a factor \(1\); if the current state is \(xs_i\), then choosing \(i\) is also a bad
decision and contributes a factor \(1\). Therefore
\[
\wt_R(x)=\wt_{R'}(x)+\wt_{R'}(xs_i).
\]

Now suppose \(\ell(xs_i)<\ell(x)\). Then from a point ending in \(X_x^\circ\),
there are \(q\) choices in the last rank-one fiber which remain in \(X_x^\circ\).
Also, from a point ending in \(X_{xs_i}^\circ\), there are \(q\) choices which move
to \(X_x^\circ\). Hence
\[
N_R(x)=qN_{R'}(x)+qN_{R'}(xs_i).
\]
Again this is exactly the recurrence in \cref{def:comb}: if the current state is
\(x\), then skipping \(i\) is a good decision and contributes a factor \(q\); if the
current state is \(xs_i\), then choosing \(i\) is a good decision and contributes a
factor \(q\). Thus
\[
\wt_R(x)=q\wt_{R'}(x)+q\wt_{R'}(xs_i).
\]

By the induction hypothesis, \(N_{R'}(z)=\wt_{R'}(z)\) for every \(z\in S_n\).
The two recurrences above therefore imply
\[
N_R(x)=\wt_R(x)
\]
for every \(x\in S_n\). Dividing by
\[
|\mathrm{BS}^R|=(q+1)^L
\]
gives
\[
\mathbb{P}_{R,q}(x)=\frac{N_R(x)}{(q+1)^L}
=
\frac{\wt_R(x)}{(q+1)^L}.
\]
This proves the proposition.
\end{proof}

Therefore, after this combinatorial interpretation, we may regard \(q\) as an indeterminate, and \(\mathbb P_{R,q}\) as a rational function of \(q\). Since a nonzero rational function has only finitely many zeros, equality of two such distributions for every prime power $q$ is equivalent to equality as rational functions of $ q$.

\begin{lemma}\label{lem:support-reduced}
Let \(R\in \Red(w)\). Then
\[
    \operatorname{supp}(\mathbb P_{R,q})=\{x\in S_n:x\le w\},
\]
where the support means the set of \(x\) for which \(\mathbb P_{R,q}(x)\) is not
identically zero as a rational function of \(q\) and $\le$ denotes the Bruhat order. 
\end{lemma}

\begin{proof}
By \cref{prop:probdist}, the support of \(\mathbb P_{R,q}\) is the set of ordinary
products of subwords of \(R\). Since \(R\) is a reduced expression for \(w\), the
subword property \cite{BjornerBrenti2005} of the Bruhat order identifies this set with
\(\{x\in S_n:x\le w\}\). The full subword gives \(w\), and \(w\) is the unique
maximal element of the Bruhat interval \([e,w]\).
\end{proof}

\begin{ex}\label{ex:comb}
    Let \(n=4\) and take
\[
R=123.
\]
We illustrate \cref{prop:probdist} by computing the weights \(\operatorname{wt}_R(x)\) directly from \cref{def:comb}.
The key observation is that at each step we maintain the total weight of subwords ending at each permutation. After normalization by \((1+q)^3\), these weights
give the distribution \(\mathbb{P}_{R,q}\).
We first compute the weights for the prefixes \(1\), \(12\), and \(123\).  Suppose
that after a prefix \(A\), a permutation \(x\) has weight \(\mathrm{wt}_A(x)\).  When
we append the next letter \(i\), there are two choices: either we do not choose
\(s_i\), so the endpoint remains \(x\), or we choose \(s_i\), so the endpoint becomes
\(xs_i\). 

If \(\ell(xs_i)>\ell(x)\),
then choosing \(s_i\) is the good decision and contributes a factor \(q\), while not
choosing it contributes a factor \(1\).  Thus in this case, the old weight
\(\mathrm{wt}_A(x)\) contributes \(\mathrm{wt}_A(x)\) to the new weight of \(x\),
and contributes \(q\,\mathrm{wt}_A(x)\) to the new weight of \(xs_i\).

If \(\ell(xs_i)<\ell(x)\), the roles are reversed.  In this case, the old weight
\(\mathrm{wt}_A(x)\) contributes \(q\,\mathrm{wt}_A(x)\) to the new weight of \(x\),
and contributes \(\mathrm{wt}_A(x)\) to the new weight of \(xs_i\).

We start from the empty word:
\[
\mathrm{wt}_{\varnothing}(1234)=1.
\]

For the first prefix \(1\), we apply \(s_1\). Since
\[
1234s_1=2134,\qquad \ell(2134)>\ell(1234),
\]
we get
\[
\mathrm{wt}_{1}(1234)=1,\qquad
\mathrm{wt}_{1}(2134)=q.
\]

Next consider the prefix \(12\). We apply \(s_2\) to the two states above:
\[
1234s_2=1324,\qquad 2134s_2=2314.
\]
Both moves increase length, so
\[
\mathrm{wt}_{12}(1234)=1,\quad
\mathrm{wt}_{12}(2134)=q,\quad
\mathrm{wt}_{12}(1324)=q,\quad
\mathrm{wt}_{12}(2314)=q^2.
\]

Finally, for the full word \(123\), we apply \(s_3\):
\[
1234s_3=1243,\qquad 2134s_3=2143,
\]
\[
1324s_3=1342,\qquad 2314s_3=2341.
\]
Again all four moves increase length. Therefore the weights for the three prefixes
are summarized in the following table:
\[
\begin{array}{c|c|c|c}
x & \mathrm{wt}_{1}(x) & \mathrm{wt}_{12}(x) & \mathrm{wt}_{123}(x) \\ \hline
1234 & 1 & 1 & 1\\
2134 & q & q & q\\
1324 & 0 & q & q\\
2314 & 0 & q^2 & q^2\\
1243 & 0 & 0 & q\\
2143 & 0 & 0 & q^2\\
1342 & 0 & 0 & q^2\\
2341 & 0 & 0 & q^3
\end{array}
\]

Thus, for \(R=123\),
\[
\begin{array}{c|c|c}
x & \mathrm{wt}_{R}(x) & \mathbb P_{R,q}(x) \\ \hline
1234 & 1 & \dfrac{1}{(1+q)^3}\\
2134 & q & \dfrac{q}{(1+q)^3}\\
1324 & q & \dfrac{q}{(1+q)^3}\\
1243 & q & \dfrac{q}{(1+q)^3}\\
2314 & q^2 & \dfrac{q^2}{(1+q)^3}\\
2143 & q^2 & \dfrac{q^2}{(1+q)^3}\\
1342 & q^2 & \dfrac{q^2}{(1+q)^3}\\
2341 & q^3 & \dfrac{q^3}{(1+q)^3}
\end{array}
\]
All other permutations in \(S_4\) have probability \(0\).

Finally,
\[
\sum_{x\in S_4}\mathrm{wt}_R(x)
=
1+3q+3q^2+q^3
=
(1+q)^3.
\]
Therefore the entries above indeed define a probability distribution after dividing
by \((1+q)^3\).
\end{ex}

For a word \(R\) , say \(R=i_1\cdots i_L\), define its generating function
\[
    F(R)=\prod_{j=1}^{L}(1+T_{i_j}),
\]
where \(T_{i_j}\) is the corresponding generator of the Hecke algebra recalled in \cref{app:hecke}. We should mention that, since Hecke Algebra is not commutative, the product is always taken in the order of the indices from left to right, that is,
\[
    F(R)=(1+T_{i_1})(1+T_{i_2})\cdots(1+T_{i_L}).
\] 

For \(x\in S_n\), set
\[
    \widetilde T_x:=q^{-\ell(x)}T_x.
\]
The elements \(\widetilde T_x\) also form a \(\mathbb Z[q^{\pm 1/2}]\)-basis of \(H_n(q)\).  We will use the following expansion.
\begin{lemma}\label{lem:hecke-expansion}
For a word \(R\), we have
\[
    F(R)=\sum_{x\in S_n}\operatorname{wt}_R(x)\widetilde T_x.
\]
\end{lemma}

\begin{proof}

We argue by induction on \(L(R)\). When \(L(R)=1\), say \(R=i_1\), the marked
subwords of \(R\) are the empty subword and the subword consisting of the unique
position. By definition,
\[
    \operatorname{wt}_R(e)=1,
    \qquad
    \operatorname{wt}_R(s_{i_1})=q,
\]
and \(\operatorname{wt}_R(w)=0\) for all other \(w\in S_n\). Hence
\[
    F(R)=1+T_{i_1}
    =
    \widetilde T_{e}+q\widetilde T_{s_{i_1}}
    =
    \sum_{x\in S_n}\operatorname{wt}_R(x)\widetilde T_x.
\]

Now assume the statement holds for all words of length \(L-1\). Let
\[
    R=R'i_L
\]
be a word of length \(L\), where \(R'\) has length \(L-1\). By the induction
hypothesis,
\[
    F(R')
    =
    \sum_{x\in S_n}\operatorname{wt}_{R'}(x)\widetilde T_x.
\]
Therefore
\[
    F(R)
    =
    F(R')(1+T_{i_L})
    =
    \left(\sum_{x\in S_n}\operatorname{wt}_{R'}(x)\widetilde T_x\right)(1+T_{i_L}).
\]

We now use the Hecke algebra multiplication rule. For \(x\in S_n\), one has
\[
    \widetilde T_x(1+T_{i_L})
    =
    \begin{cases}
    \widetilde T_x+q\widetilde T_{xs_{i_L}},
    & \text{if } \ell(xs_{i_L})>\ell(x),\\[4pt]
    q\widetilde T_x+\widetilde T_{xs_{i_L}},
    & \text{if } \ell(xs_{i_L})<\ell(x).
    \end{cases}
\]
Indeed, if \(\ell(xs_{i_L})>\ell(x)\), then
\[
    T_xT_{i_L}=T_{xs_{i_L}},
\]
so
\[
    \widetilde T_x(1+T_{i_L})
    =
    \widetilde T_x+q\widetilde T_{xs_{i_L}}.
\]
If \(\ell(xs_{i_L})<\ell(x)\), then
\[
    T_xT_{i_L}=(q-1)T_x+qT_{xs_{i_L}},
\]
and hence
\[
    \widetilde T_x(1+T_{i_L})
    =
    q\widetilde T_x+\widetilde T_{xs_{i_L}}.
\]

On the other hand, the same two cases are exactly the recursive rule defining
\(\operatorname{wt}_R\). Indeed, every  subword of \(R=R'i_L\) is obtained
uniquely from a subword of \(R'\), together with the decision of whether or
not to select the last position. Suppose a subword of \(R'\) evaluates to
\(x\).

If \(\ell(xs_{i_L})>\ell(x)\), then not selecting the last position keeps the
evaluation equal to \(x\) and contributes a factor \(1\), while selecting the last
position changes the evaluation to \(xs_{i_L}\) and contributes a factor \(q\).
Thus this case contributes
\[
    \widetilde T_x+q\widetilde T_{xs_{i_L}}.
\]

If \(\ell(xs_{i_L})<\ell(x)\), then not selecting the last position keeps the
evaluation equal to \(x\) and contributes a factor \(q\), while selecting the last
position changes the evaluation to \(xs_{i_L}\) and contributes a factor \(1\).
Thus this case contributes
\[
    q\widetilde T_x+\widetilde T_{xs_{i_L}}.
\]

Therefore multiplication by \(1+T_{i_L}\) produces exactly the same recursion as
the weighted enumeration of subwords after adjoining the last letter
\(i_L\). Consequently,
\[
    F(R)
    =
    \sum_{x\in S_n}\operatorname{wt}_R(x)\widetilde T_x.
\]

\end{proof}

\begin{cor}\label{cor:hecke-characterization}
Let \(R\) and \(R'\) be two words with \(L(R)=L(R')\). Then
\(R\equiv R'\) if and only if
\[
    F(R)=F(R'),
\]
equivalently,
\[
    \prod_{i\in R}(1+T_i)=\prod_{i\in R'}(1+T_i).
\]
\end{cor}

\begin{proof}
This follows directly from \cref{lem:hecke-expansion}, since two words of the same
length have the same normalizing denominator \((1+q)^{L(R)}\) in the definition
of \(\mathbb P_{R,q}\).
\end{proof}

\begin{defin}
Let \(w\in S_n\). We define the \emph{standard representing word} of \(w\), denoted by
\[
\operatorname{Std}(w),
\]
to be the word obtained from the reverse bubble sort procedure applied to \(w\) as follows.

Starting from the one-line notation of \(w\), we repeatedly perform left-to-right passes.
In each pass, whenever two adjacent entries are in increasing order, we swap them.
Equivalently, if at some step the current permutation is \(u\in S_n\) and
\[
u(i)<u(i+1),
\]
then we apply the simple reflection \(s_i\), recording the index \(i\) in the word.
We continue this process until the longest element
\[
w_0=n(n-1)\cdots 321
\]
is reached.

The resulting sequence of indices
\[
i_1i_2\cdots i_{\ell}
\]
is called the standard representing word of \(w\). Thus
\[
\operatorname{Std}(w)=i_1i_2\cdots i_{\ell},
\]
where
\[
w s_{i_1}s_{i_2}\cdots s_{i_{\ell}}=w_0,
\]
and the indices are recorded in the order in which the adjacent swaps occur in the reverse bubble sort procedure.

Equivalently, \(\operatorname{Std}(w)\) is determined recursively as follows.
If \(w=w_0\), then
\[
\operatorname{Std}(w)=\emptyset.
\]
Otherwise, starting from \(w\), we repeatedly scan the permutation from left to right; whenever
\[
u(i)<u(i+1),
\]
we replace \(u\) by \(us_i\) and append the letter \(i\) to the word. Repeating this procedure until \(w_0\) is obtained produces \(\operatorname{Std}(w)\).

In other words, \(\operatorname{Std}(w)\) is the word naturally associated with sorting \(w\) to \(w_0\) by reverse bubble sort.
\end{defin}

\begin{defin}\label{def:std-ext}
For any reduced word \(R\in\Red(w)\), define its \emph{standard extension} to be
\[
    \Ext(R):=R\std(w).
\]
Since each step in the construction of \(\std(w)\) increases Coxeter length by one,
\(\Ext(R)\) is a reduced word for \(w_0\).
\end{defin}

Now we can reduce the \cref{quest:main} to the case where \(w(R_1)=w(R_2)=w_0\).

\begin{proposition}\label{prop:standard-extension-preserves-equivalence}
Let \(R,R'\in\Red(w)\). If \(R\equiv R'\), then
\(\Ext(R)\equiv\Ext(R')\).
\end{proposition}

\begin{proof}
Assume that \(R\equiv R'\). By \cref{cor:hecke-characterization}, we have
\[
    \prod_{i\in R}(1+T_i)=\prod_{i\in R'}(1+T_i).
\]
Multiplying both sides on the right by the common factor
\(\prod_{i\in\std(w)}(1+T_i)\), we obtain
\[
    \prod_{i\in R\std(w)}(1+T_i)
    =
    \prod_{i\in R'\std(w)}(1+T_i).
\]
By \cref{cor:hecke-characterization}, this is equivalent to
\(\Ext(R)\equiv\Ext(R')\).
\end{proof}

\begin{proposition}\label{prop:shorten}
    Let $R,R',T$ be three fixed words satisfying both $RT$ and $R'T$ are reduced, and $RT \sim R'T $. Then $R \sim R'$.
\end{proposition}

\begin{proof}
This follows directly from \cref{lem:commutation-characterization}.
\end{proof}

Now we can prove the following reduction property:

\begin{cor}\label{cor:reduction-w0}
To prove \Cref{thm:main} for arbitrary reduced words, it suffices to prove it in
the special case
\[
    R_1,R_2\in\Red(w_0).
\]
\end{cor}

\begin{proof}
It suffices to prove the nontrivial direction. Suppose that
\[
    \mathbb P_{R_1,q}=\mathbb P_{R_2,q}
\]
as rational functions of \(q\). By \cref{lem:support-reduced}, the support of
\(\mathbb P_{R,q}\) has the unique Bruhat-maximal element \(w(R)\). Hence equality
of the two distributions implies equality of supports, and therefore
\[
    w(R_1)=w(R_2)=:w.
\]

Assume that \Cref{thm:main} has been proved for reduced words of \(w_0\).  By
\cref{prop:standard-extension-preserves-equivalence}, we have
\[
    \mathbb P_{\Ext(R_1),q}=\mathbb P_{\Ext(R_2),q}
\]
as rational functions of \(q\). Since \(\Ext(R_1),\Ext(R_2)\in\Red(w_0)\), the
\(w_0\)-case of \Cref{thm:main} gives
\[
    \Ext(R_1)\sim \Ext(R_2).
\]
Finally, by \cref{prop:shorten}, we obtain \(R_1\sim R_2\). This proves the
reduction to the case of reduced words of \(w_0\).
\end{proof}


The next arguments first treat reduced words of \(w_0\). The general case will be recovered in \cref{subsec:main-theorem-proof} using the reduction above.

 To see how a braid move affects our distribution, it suffices to reveal how it affects the expansion $F(R)$. 

 Observe that for
\[
    R_1=T_1\,i\,(i+1)\,i\,T_2,
    \qquad
    R_2=T_1\,(i+1)\,i\,(i+1)\,T_2,
\]
denote
\[
    R_1'=T_1\,i\,T_2,
    \qquad
    R_2'=T_1\,(i+1)\,T_2,
\]
and
\[
    v_1=\Dem(T_1\,i\,T_2),
    \qquad
    v_2=\Dem(T_1\,(i+1)\,T_2),
\]
we have
\[
    F(R_1')-F(R_2')
    =
    q^{N-2}\bigl(\widetilde T_{v_1}-\widetilde T_{v_2}\bigr)
    +
    O(q^{N-3}).
\]
Moreover, one can check that
\[
    (1+T_i)(1+T_{i+1})(1+T_i)
    -
    (1+T_{i+1})(1+T_i)(1+T_{i+1})
    =
    q\bigl((1+T_i)-(1+T_{i+1})\bigr),
\]
therefore,
\begin{equation}\label{eq:braid-diff}
\begin{aligned}
    F(R_1)-F(R_2)
    &=
    q\bigl(F(R_1')-F(R_2')\bigr) \\
    &=
    q^{N-1}\bigl(\widetilde T_{v_1}-\widetilde T_{v_2}\bigr)
    +
    O(q^{N-2}).
\end{aligned}
\end{equation}

Here \(O(q^{N-2})\) denotes a linear combination
\[
    \sum_{x\in S_n} \wt_x(q)\,\widetilde T_x
\]
such that each \(\wt_x(q)\in \mathbb{Z}[q]\) has degree at most \(N-2\).

\begin{defin}
For a reduced word $R\in \Red(w_0)$, define
\[
D_R(x):=[q^{N-1}]\wt_R(x)\qquad \text{for } x\in S_n,
\]
where $N=\binom{n}{2}$.
\end{defin}

\begin{remark}[Combinatorial meaning of \(D_R(x)\)]
Let \(R=i_1i_2\cdots i_N\in\Red(w_0)\). Then \(D_R(x)\) counts one-letter Demazure shortenings of \(R\). More precisely,
\[
D_R(x)
=
\#\left\{
j\in[N]:
\Dem(i_1\cdots\widehat{i_j}\cdots i_N)=x
\right\}.
\]

Indeed, multiply \(F(R)\) from left to right in the normalized basis. At each step, the degree-maximizing choice is exactly the Demazure choice and contributes one factor of \(q\). Since \(R\) is reduced, choosing the degree-maximizing branch at every position gives the unique top-degree term
\[
q^N\widetilde T_{w_0}.
\]
A contribution in degree \(N-1\) is obtained by making the non-maximal choice at exactly one position and the maximal choice at all other positions. If the non-maximal choice is made at position \(j\), then the resulting endpoint is precisely
\[
\Dem(i_1\cdots\widehat{i_j}\cdots i_N).
\]
Thus each deletion position \(j\) contributes one copy of \(q^{N-1}\) to the coefficient of this Demazure product, and summing over all \(j\) gives the stated formula.
\end{remark}

The preceding discussion shows that the coefficient \(D_R(x)\) is governed by one-letter Demazure shortenings of \(R\).  In order to compare these coefficients for two reduced words related by a braid move, we therefore need a precise description of the Demazure products obtained by shortening the braid.  This is the purpose of \cref{subsec:demazure-shortenings}.

\subsection{Demazure shortenings and forward--backward decompositions}\label{subsec:demazure-shortenings}

We begin with the local calculation used throughout the rest of the proof.  Suppose that a reduced word contains a braid \(i(i+1)i\), and shorten this braid to the single letter \(i\).  The following proposition describes the resulting Demazure product by comparing it with the longest element \(w_0\).

\begin{proposition}\label{prop:cycle-description}
Let $R_1=T_1\,i\,(i+1)\,i\,T_2$ be a reduced word for the longest element \(w_0\in S_n\). Denote $v=\Dem (T_1\, i\, T_2)$. Let \(u=w(T_1)\) be the ordinary product of the prefix \(T_1\), and write
\[
u(i)=a,\qquad u(i+1)=b,\qquad u(i+2)=c.
\]
Then \(a<b<c\). 
Define
\[
A:=\{u(j):\,j\le i+1,\ u(j)\in [a,c]\},
\qquad
B:=\{u(j):\,j\ge i+2,\ u(j)\in [a,c]\}.
\]
Write these sets in increasing order as
\[
A=\{a=g_1<g_2<\cdots<g_r\},
\qquad
B=\{h_1<h_2<\cdots<h_s=c\}.
\]
Let \(\sigma\in S_n\) be the cycle
\[
\sigma=(g_1\,g_2\,\cdots\,g_r\,h_s\,h_{s-1}\,\cdots\,h_1),
\]
extended by the identity outside \([a,c]\). Then
\[
w_0=\sigma v.
\]
Equivalently, with the notation \(\pos_x(t):=x^{-1}(t)\), one has
\begin{align*}
    \pos_v(g_j)&=\pos_{w_0}(g_{j+1})\qquad (1\le j<r),\\
    \pos_v(g_r)&=\pos_{w_0}(h_s),\\
    \pos_v(h_j)&=\pos_{w_0}(h_{j-1})\qquad (1<j\le s),\\
    \pos_v(h_1)&=\pos_{w_0}(g_1).
\end{align*}
\end{proposition}

Before giving the proof, we illustrate the mechanism behind
\cref{prop:cycle-description} in a concrete example. The diagram below should be
read from top to bottom. Each vertical arrow records one explicit block of
simple reflections applied from left to right to both rows, in the Demazure
sense. The proof below abstracts the same three stages in general.

\begin{ex}
\begingroup
\setlength{\abovedisplayskip}{6pt}
\setlength{\belowdisplayskip}{6pt}
\setlength{\abovedisplayshortskip}{4pt}
\setlength{\belowdisplayshortskip}{4pt}
Let
\[
\begin{gathered}
n=10,\qquad i=5,\qquad
u=9\,7\,6\,2\,1\,4\,10\,8\,5\,3,\\
a=u(5)=1,\qquad b=u(6)=4,\qquad c=u(7)=10.
\end{gathered}
\]
Thus
\[
A=\{1,2,4,6,7,9\},\qquad B=\{3,5,8,10\}.
\]
Writing
\[
\begin{aligned}
A&=\{g_1<g_2<\cdots<g_6\}=\{1,2,4,6,7,9\},\\
B&=\{h_1<h_2<h_3<h_4\}=\{3,5,8,10\},
\end{aligned}
\]
the cycle predicted by \cref{prop:cycle-description} is
\[
\sigma=(g_1\,g_2\,g_3\,g_4\,g_5\,g_6\,h_4\,h_3\,h_2\,h_1)
      =(1\,2\,4\,6\,7\,9\,10\,8\,5\,3).
\]

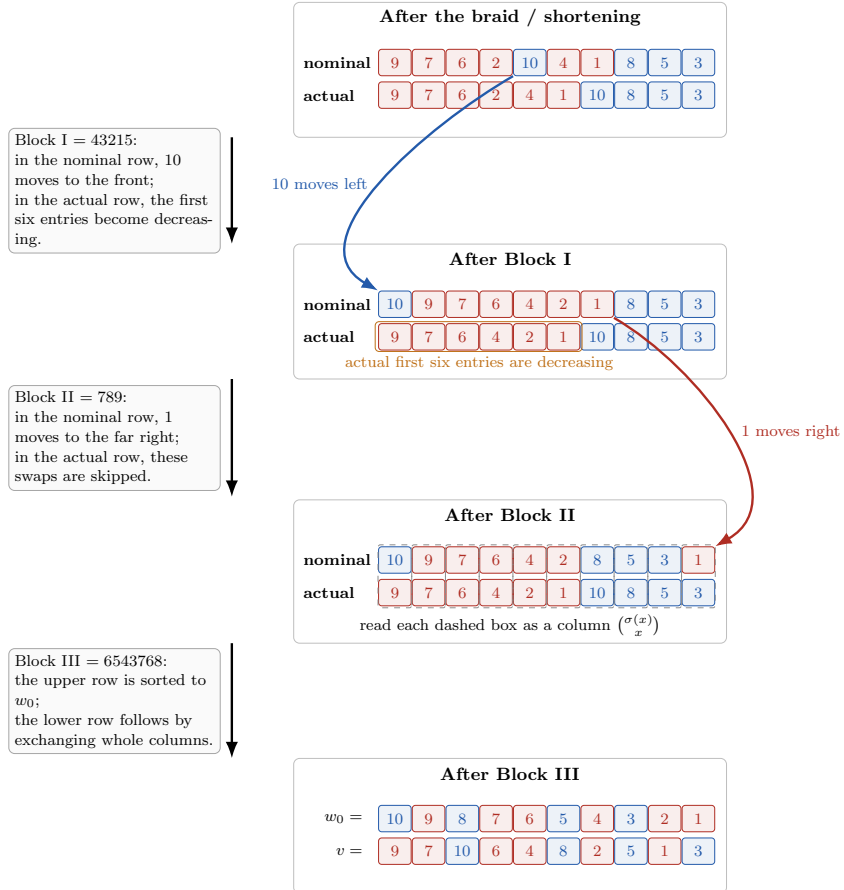
\begin{center}
\begingroup
\definecolor{Ared}{RGB}{175,45,35}
\definecolor{Bblue}{RGB}{35,90,170}
\definecolor{NoteOrange}{RGB}{190,110,20}

\resizebox{!}{0.54\textheight}{%
\begin{tikzpicture}[
    >=Latex,
    x=1cm,y=1cm,
    cellA/.style={draw=Ared, fill=Ared!8, rounded corners=1.6pt,
        minimum width=5.8mm, minimum height=4.7mm, inner sep=0.5pt,
        font=\footnotesize, text=Ared},
    cellB/.style={draw=Bblue, fill=Bblue!8, rounded corners=1.6pt,
        minimum width=5.8mm, minimum height=4.7mm, inner sep=0.5pt,
        font=\footnotesize, text=Bblue},
    rowlab/.style={font=\footnotesize\bfseries, anchor=west},
    panel/.style={draw=black!25, rounded corners=3pt},
    paneltitle/.style={font=\small\bfseries},
    blocknote/.style={draw=black!35, rounded corners=3pt, fill=black!2,
        font=\footnotesize, align=left, text width=3.55cm, inner sep=3pt},
    pairbox/.style={draw=black!45, dashed, rounded corners=1.2pt},
    flow/.style={->, very thick},
    emphbox/.style={draw=NoteOrange, rounded corners=2pt}
]

\path[use as bounding box] (-5.25,-15.8) rectangle (6.70,0.80);

\newcommand{\Acell}[3]{\node[cellA] at ({0.55 + 0.60*(#1)},#2) {#3};}
\newcommand{\Bcell}[3]{\node[cellB] at ({0.55 + 0.60*(#1)},#2) {#3};}
\newcommand{\RowLab}[2]{\node[rowlab] at (-1.2,#1) {#2};}
\newcommand{\RowLabShort}[2]{\node[rowlab, anchor=east] at (0.1,#1) {#2};}

\def\Y{0.00}
\draw[panel] (-1.25,\Y+0.45) rectangle (6.45,\Y-1.95);
\node[paneltitle] at (2.60,\Y+0.18) {After the braid / shortening};

\RowLab{\Y-0.62}{nominal}
\Acell{0}{\Y-0.62}{9}
\Acell{1}{\Y-0.62}{7}
\Acell{2}{\Y-0.62}{6}
\Acell{3}{\Y-0.62}{2}
\node[cellB] (p0ten) at ({0.55 + 0.60*4},\Y-0.62) {10};
\Acell{5}{\Y-0.62}{4}
\Acell{6}{\Y-0.62}{1}
\Bcell{7}{\Y-0.62}{8}
\Bcell{8}{\Y-0.62}{5}
\Bcell{9}{\Y-0.62}{3}

\RowLab{\Y-1.20}{actual}
\Acell{0}{\Y-1.20}{9}
\Acell{1}{\Y-1.20}{7}
\Acell{2}{\Y-1.20}{6}
\Acell{3}{\Y-1.20}{2}
\Acell{4}{\Y-1.20}{4}
\Acell{5}{\Y-1.20}{1}
\Bcell{6}{\Y-1.20}{10}
\Bcell{7}{\Y-1.20}{8}
\Bcell{8}{\Y-1.20}{5}
\Bcell{9}{\Y-1.20}{3}

\def\Y{-4.30}
\draw[panel] (-1.25,\Y+0.45) rectangle (6.45,\Y-1.95);
\node[paneltitle] at (2.60,\Y+0.18) {After Block I};

\RowLab{\Y-0.62}{nominal}
\node[cellB] (p1ten) at ({0.55 + 0.60*0},\Y-0.62) {10};
\Acell{1}{\Y-0.62}{9}
\Acell{2}{\Y-0.62}{7}
\Acell{3}{\Y-0.62}{6}
\Acell{4}{\Y-0.62}{4}
\Acell{5}{\Y-0.62}{2}
\node[cellA] (p1one) at ({0.55 + 0.60*6},\Y-0.62) {1};
\Bcell{7}{\Y-0.62}{8}
\Bcell{8}{\Y-0.62}{5}
\Bcell{9}{\Y-0.62}{3}

\RowLab{\Y-1.20}{actual}
\Acell{0}{\Y-1.20}{9}
\Acell{1}{\Y-1.20}{7}
\Acell{2}{\Y-1.20}{6}
\Acell{3}{\Y-1.20}{4}
\Acell{4}{\Y-1.20}{2}
\Acell{5}{\Y-1.20}{1}
\Bcell{6}{\Y-1.20}{10}
\Bcell{7}{\Y-1.20}{8}
\Bcell{8}{\Y-1.20}{5}
\Bcell{9}{\Y-1.20}{3}

\draw[emphbox, NoteOrange] (0.20,\Y-1.46) rectangle (3.88,\Y-0.93);
\node[font=\footnotesize, text=NoteOrange]
    at (2.05,\Y-1.67) {actual first six entries are decreasing};

\def\Y{-8.85}
\draw[panel] (-1.25,\Y+0.45) rectangle (6.45,\Y-2.10);
\node[paneltitle] at (2.60,\Y+0.18) {After Block II};

\RowLab{\Y-0.62}{nominal}
\Bcell{0}{\Y-0.62}{10}
\Acell{1}{\Y-0.62}{9}
\Acell{2}{\Y-0.62}{7}
\Acell{3}{\Y-0.62}{6}
\Acell{4}{\Y-0.62}{4}
\Acell{5}{\Y-0.62}{2}
\Bcell{6}{\Y-0.62}{8}
\Bcell{7}{\Y-0.62}{5}
\Bcell{8}{\Y-0.62}{3}
\node[cellA] (p2one) at ({0.55 + 0.60*9},\Y-0.62) {1};

\RowLab{\Y-1.20}{actual}
\Acell{0}{\Y-1.20}{9}
\Acell{1}{\Y-1.20}{7}
\Acell{2}{\Y-1.20}{6}
\Acell{3}{\Y-1.20}{4}
\Acell{4}{\Y-1.20}{2}
\Acell{5}{\Y-1.20}{1}
\Bcell{6}{\Y-1.20}{10}
\Bcell{7}{\Y-1.20}{8}
\Bcell{8}{\Y-1.20}{5}
\Bcell{9}{\Y-1.20}{3}

\foreach \k in {0,...,9}{
    \draw[pairbox] ({0.55 + 0.60*(\k) - 0.30},\Y-1.47)
                   rectangle
                   ({0.55 + 0.60*(\k) + 0.30},\Y-0.35);
}

\node[font=\footnotesize, align=center]
    at (2.60,\Y-1.78)
    {read each dashed box as a column \(\binom{\sigma(x)}{x}\)};

\def\Y{-13.45}
\draw[panel] (-1.25,\Y+0.45) rectangle (6.45,\Y-1.95);
\node[paneltitle] at (2.60,\Y+0.18) {After Block III};

\RowLabShort{\Y-0.62}{\(w_0=\)}
\Bcell{0}{\Y-0.62}{10}
\Acell{1}{\Y-0.62}{9}
\Bcell{2}{\Y-0.62}{8}
\Acell{3}{\Y-0.62}{7}
\Acell{4}{\Y-0.62}{6}
\Bcell{5}{\Y-0.62}{5}
\Acell{6}{\Y-0.62}{4}
\Bcell{7}{\Y-0.62}{3}
\Acell{8}{\Y-0.62}{2}
\Acell{9}{\Y-0.62}{1}

\RowLabShort{\Y-1.20}{\(v=\)}
\Acell{0}{\Y-1.20}{9}
\Acell{1}{\Y-1.20}{7}
\Bcell{2}{\Y-1.20}{10}
\Acell{3}{\Y-1.20}{6}
\Acell{4}{\Y-1.20}{4}
\Bcell{5}{\Y-1.20}{8}
\Acell{6}{\Y-1.20}{2}
\Bcell{7}{\Y-1.20}{5}
\Acell{8}{\Y-1.20}{1}
\Bcell{9}{\Y-1.20}{3}

\draw[flow] (-2.35,-1.95) -- (-2.35,-3.85);
\node[blocknote, anchor=east] at (-2.55,-2.90)
{Block I \(=43215\):\\
in the nominal row, \(10\) moves to the front;\\
in the actual row, the first six entries become decreasing.};

\draw[flow] (-2.35,-6.25) -- (-2.35,-8.35);
\node[blocknote, anchor=east] at (-2.55,-7.30)
{Block II \(=789\):\\
in the nominal row, \(1\) moves to the far right;\\
in the actual row, these swaps are skipped.};

\draw[flow] (-2.35,-10.95) -- (-2.35,-13.00);
\node[blocknote, anchor=east] at (-2.55,-11.98)
{Block III \(=6543768\):\\
the upper row is sorted to \(w_0\);\\
the lower row follows by exchanging whole columns.};

\draw[->, very thick, Bblue]
    (p0ten.south west) to[out=210,in=150]
    node[left=1mm, font=\footnotesize, text=Bblue] {\(10\) moves left}
    (p1ten.north west);

\draw[->, very thick, Ared]
    (p1one.south east) to[out=330,in=30]
    node[right=1mm, font=\footnotesize, text=Ared] {\(1\) moves right}
    (p2one.north east);

\end{tikzpicture}
}%
\endgroup
\end{center}

The picture should be read as follows. Block I does two things: in the nominal
row it moves \(10\) to the front, while in the actual row the same Demazure
operations make the first six entries decreasing. Block II moves \(1\) to the far
right in the nominal row, and the corresponding attempted swaps are skipped in
the actual row. After Block II, the two rows are read columnwise; each dashed box
has the form
\[
\begin{pmatrix}\sigma(x)\\ x\end{pmatrix}.
\]
Finally, Block III sorts the upper row to \(w_0\). During this last stage, the
lower row follows by exchanging neighboring columns, while the columns themselves
are never broken. Hence the same column correspondence survives to the end, and
we obtain
\[
w_0=\sigma v.
\]
In particular, the values in \(A\) move forward relative to their positions in
\(w_0\), while the values in \(B\) move backward.
\endgroup
\end{ex}

The example isolates the two structural features used in the proof. After
Blocks I and II, the relevant information is the column pairing between the
nominal row and the actual row: each column has the form
\[
\begin{pmatrix}
\sigma(x)\\
x
\end{pmatrix}.
\]
Block III then sorts the nominal row to \(w_0\). During this last stage, each
adjacent swap in the nominal row induces the corresponding adjacent swap in the
actual row, so the columns are exchanged as units and are never broken. The proof
below formalizes this mechanism for an arbitrary braid: first we construct the
general analogues of Blocks I and II, and then we show that the remaining suffix
preserves the resulting column correspondence.

We now prove \cref{prop:cycle-description} in general.

\begin{proof}[Proof of \cref{prop:cycle-description}]
Since \(T_1\,i(i+1)i\) is reduced, the three right multiplications by
\(s_i,s_{i+1},s_i\) are all length-increasing. Hence
\[
u(i)<u(i+1),\qquad u(i)<u(i+2),\qquad u(i+1)<u(i+2),
\]
and therefore \(a<b<c\).

Throughout the proof, whenever a nominal row \(\widetilde x\) and an actual row
\(x\) are displayed together, we compare them position by position. The
\emph{column} at position \(p\) means the ordered pair
\[
\begin{pmatrix}
\widetilde x(p)\\
x(p)
\end{pmatrix}.
\]
We say that a column is preserved through a simultaneous Demazure step if,
after applying the same simple reflection to the two rows, its two entries
remain paired. Equivalently, the step either leaves the column in place or
exchanges it as a whole with a neighboring column.

We first reduce to the interval \([a,c]\). Immediately after replacing the full
braid by its shortening, every value \(d\notin[a,c]\) occurs in the same position
in the nominal row and in the actual row. Thus it forms a singleton column
\[
\begin{pmatrix}
d\\
d
\end{pmatrix}.
\]
Moreover, \(d<a\) or \(d>c\). Hence \(d\) has the same order relation with every
value of \([a,c]\) in both rows. It follows that, whenever a simultaneous
Demazure step involves the column \(\binom d d\) and a neighboring column, the
swap is length-increasing in the nominal row if and only if it is
length-increasing in the actual row. Therefore such singleton columns are
preserved throughout the process. Consequently, values outside \([a,c]\) end in
their \(w_0\)-positions, and only values in \([a,c]\) can be displaced.

After relabelling the interval \([a,c]\) by \([n]\), we may assume
\[
a=1,\qquad c=n,\qquad g_1=1,\qquad h_s=n,\qquad A\sqcup B=[n].
\]
Then \(r=i+1\). Let \(t\) be determined by \(g_t=b\).

We next explain how the suffix may be chosen. We shall use the following
elementary extension principle. Suppose \(y\in S_n\), and suppose a word \(S\)
is length-increasing when applied successively to \(y\). If \(yS=z\), then \(S\)
can be extended to a reduced word from \(y\) to \(w_0\). Indeed,
\[
L(S)+\ell(z^{-1}w_0)
=
(\ell(z)-\ell(y))+(N-\ell(z))
=
N-\ell(y)
=
\ell(y^{-1}w_0).
\]
Thus \(S\), followed by any reduced expression for \(z^{-1}w_0\), is a reduced
expression for \(y^{-1}w_0\).

Apply this principle to the nominal row after the full braid. First perform
reverse bubble sort inside the two side blocks, namely in positions
\(1,\ldots,i-1\) and \(i+3,\ldots,n\), until both side blocks are decreasing.
Every swap used in this sorting exchanges an adjacent increasing pair, so it is
length-increasing. Hence this sorting word may be chosen as an initial segment
of a reduced suffix. It uses only generators inside the side blocks, so it acts
identically on the actual row. Since the Demazure product by a fixed group
element is independent of the chosen reduced expression, replacing \(T_2\) by
this new reduced suffix does not change \(v\).

After this preliminary normalization, write
\[
\alpha=g_r g_{r-1}\cdots \widehat{g_t}\cdots g_2,
\qquad
\beta=h_{s-1}h_{s-2}\cdots h_1.
\]
The two rows from which the remaining comparison starts are
\[
\widetilde x_0=\alpha\,h_s\,g_t\,g_1\,\beta,
\qquad
x_0=\alpha\,g_t\,g_1\,h_s\,\beta.
\]
Here \(\widetilde x_0\) is the nominal row, obtained by keeping the full braid,
and \(x_0\) is the actual row, obtained after shortening.

Now define
\[
P_1=(i-1)(i-2)\cdots 1\,i(i-1)\cdots(r-t+2),
\qquad
P_2=(r+1)(r+2)\cdots(n-1).
\]
All arrows below are Demazure arrows. The block \(P_1\) gives
\[
\widetilde x_0
\xrightarrow{P_1}
\widetilde x_1
=
h_s\,g_r\,g_{r-1}\cdots g_2\mid g_1\,h_{s-1}\cdots h_1,
\]
while
\[
x_0
\xrightarrow{P_1}
x_1
=
g_r\,g_{r-1}\cdots g_2\,g_1\mid h_s\,h_{s-1}\cdots h_1.
\]
Then \(P_2\) moves \(g_1=1\) to the far right in the nominal row and is skipped
in the actual row:
\[
\widetilde x_1
\xrightarrow{P_2}
\widetilde x_2
=
h_s\,g_r\,g_{r-1}\cdots g_2\mid h_{s-1}\cdots h_1\,g_1,
\]
whereas
\[
x_1
\xrightarrow{P_2}
x_2
=
g_r\,g_{r-1}\cdots g_2\,g_1\mid h_s\,h_{s-1}\cdots h_1.
\]
The displayed nominal moves are length-increasing. Therefore, by the extension
principle above, we may choose the remaining suffix after the preliminary
normalization to have the form
\[
P_1P_2P_3
\]
for some reduced word \(P_3\), with \(P_3\) carrying \(\widetilde x_2\) to
\(w_0\).

We now form the columns of the two-row array with top row \(\widetilde x_2\) and
bottom row \(x_2\):
\[
\begin{array}{c|ccccccccc}
\widetilde x_2
& h_s & g_r & g_{r-1} & \cdots & g_2
& h_{s-1} & h_{s-2} & \cdots & g_1\\
x_2
& g_r & g_{r-1} & g_{r-2} & \cdots & g_1
& h_s & h_{s-1} & \cdots & h_1 .
\end{array}
\]
Equivalently, every column has the form
\[
\begin{pmatrix}
\sigma(x)\\
x
\end{pmatrix}.
\]

It remains to check that \(P_3\) preserves these columns. The nominal row
\(\widetilde x_2\) is the concatenation of two decreasing blocks,
\[
h_s\,g_r\,g_{r-1}\cdots g_2
\qquad\text{and}\qquad
h_{s-1}\cdots h_1\,g_1.
\]
Hence every remaining adjacent swap in the nominal row exchanges a neighboring
pair \(g_j,h_k\) with
\[
2\le j\le r,\qquad 1\le k\le s-1,\qquad g_j<h_k.
\]
In the actual row, the corresponding lower entries are \(g_{j-1}\) and
\(h_{k+1}\). Since
\[
g_{j-1}<g_j<h_k<h_{k+1},
\]
the same adjacent swap is also length-increasing in the actual row. Thus each
letter of \(P_3\) exchanges the same two neighboring columns in both rows; in
particular, no letter separates the two entries of any column.

By induction over the letters of \(P_3\), the column correspondence survives
until the nominal row becomes \(w_0\). The final actual row is \(v\). Therefore
the final comparison between \(w_0\) and \(v\) is still given by columns
\(\binom{\sigma(x)}{x}\), which means exactly that
\[
w_0=\sigma v.
\]
Equivalently,
\begin{align*}
    \pos_v(g_j)&=\pos_{w_0}(g_{j+1})\qquad (1\le j<r),\\
    \pos_v(g_r)&=\pos_{w_0}(h_s),\\
    \pos_v(h_j)&=\pos_{w_0}(h_{j-1})\qquad (1<j\le s),\\
    \pos_v(h_1)&=\pos_{w_0}(g_1).
\end{align*}
This proves the proposition.
\end{proof}

The position formula in \cref{prop:cycle-description} suggests a more compact way to record how a Demazure shortening differs from \(w_0\).  We introduce the following notation.

\begin{defin}\label{def:supp_fwd_bwd}
Denote by $\pos_x(t)$ the position of value $t$ in the one-line form of permutation $x$,
i.e.\ $\pos_x(t)=x^{-1}(t)$.
Define
\[
\Supp(x):=\{t\in [n]\mid \pos_x(t)\neq \pos_{w_0}(t)\}.
\]
Define ``forward/backward positions'' by
\[
\Fwd(x):=\{t\in [n]\mid \pos_x(t)<\pos_{w_0}(t)\},
\]
\[
\Bwd(x):=\{t\in [n]\mid \pos_x(t)>\pos_{w_0}(t)\}.
\]
Then $\Supp(x)=\Fwd(x)\sqcup \Bwd(x)$.
\end{defin}

With this notation, \cref{prop:cycle-description} has the following immediate consequence.  It gives the support and the forward--backward decomposition for the two Demazure shortenings associated with a braid move.

\begin{lemma}\label{lem:key-obs}
For $R_1=T_1\, i\, (i+1)\, i\, T_2$, $R_2=T_1\, (i+1)\, i\, (i+1)\, T_2$, denote
\[
v_1=\Dem(T_1\, i\, T_2),\qquad v_2=\Dem(T_1\, (i+1)\, T_2).
\]
Let \(u=w(T_1)\), and write
\[
u(i)=a,\qquad u(i+1)=b,\qquad u(i+2)=c.
\]
Then $a<b<c$, and
\[
\Supp(v_1)=\Supp(v_2)=[a,c]\cap \mathbb{Z}.
\]
Moreover,
\begin{align*}
\Fwd(v_1) &= \{u(j)\mid j\le i+1,\ u(j)\in [a,c]\}\supseteq \{a,b\},\\
\Bwd(v_1) &= \{u(j)\mid j\ge i+2,\ u(j)\in [a,c]\}\supseteq \{c\},\\
\Fwd(v_2) &= \{u(j)\mid j\le i,\ u(j)\in [a,c]\}\supseteq \{a\},\\
\Bwd(v_2) &= \{u(j)\mid j\ge i+1,\ u(j)\in [a,c]\}\supseteq \{b,c\}.
\end{align*}
\end{lemma}

\begin{proof}
For \(v_1\), the statement follows directly from \cref{prop:cycle-description}: the values in
\[
\{u(j):j\le i+1,\ u(j)\in[a,c]\}
\]
move forward, while the values in
\[
\{u(j):j\ge i+2,\ u(j)\in[a,c]\}
\]
move backward. Hence these are precisely \(\Fwd(v_1)\) and \(\Bwd(v_1)\), and the support is \([a,c]\cap\mathbb Z\).

The formulas for \(v_2\) are obtained in the same way from the other side of the braid, namely by shortening \((i+1)i(i+1)\) to \(i+1\). The stated inclusions are immediate from \(u(i)=a\), \(u(i+1)=b\), and \(u(i+2)=c\).
\end{proof}

As a first consequence, the two Demazure shortenings produced by the two sides of the braid move cannot coincide.

\begin{cor}
    The notations are the same as above. Then $v_1\ne v_2$ as elements in $S_n$.
\end{cor}

\begin{proof}
By \cref{lem:key-obs}, the value \(b\) lies in \(\Fwd(v_1)\) and in \(\Bwd(v_2)\). Hence
\(\pos_{v_1}(b)<\pos_{w_0}(b)<\pos_{v_2}(b)\), so \(v_1\ne v_2\).
\end{proof}

Thus a braid move produces two distinct Demazure shortenings with the same support interval \([a,c]\), while the middle value \(b\) changes from a forward value to a backward value.  This local sign change is the information encoded by the statistic \(\Theta_R\) in \cref{subsec:theta-statistic}.


\subsection{The signed interval statistic \texorpdfstring{\(\Theta_R\)}{Theta R}}\label{subsec:theta-statistic}

\Cref{subsec:demazure-shortenings} shows that the local contribution of a braid move is detected by two pieces of information: the support interval of the Demazure shortening, and whether the middle value moves forward or backward. We now aggregate this signed information over all second-highest contributions with a fixed support interval.

\begin{defin}\label{def:sigma}
    For \(x\in S_n\) and \(b\in[n]\), define
    \[
    \sigma_b(x)=\sgn (\pos_x(b)-\pos_{w_0}(b))=
    \begin{cases}
    -1,& b\in\Fwd(x),\\
    +1,& b\in\Bwd(x),\\
    0,& b\notin\Supp(x).
    \end{cases}
    \]
\end{defin}

\begin{defin}\label{def:Theta}
For \(R\in\Red(w_0)\) and \((a,b,c)\in \mathcal{T}_n\), define
\[
\Theta_R(a,b,c)
:=
\sum_{\substack{x\in S_n\\ \Supp(x)=[a,c]\cap\mathbb Z}}
D_R(x)\sigma_b(x).
\]
\end{defin}

Thus \(\Theta_R(a,b,c)\) is the signed total contribution of all one-letter Demazure shortenings whose displaced values are exactly the interval \([a,c]\), with the sign determined by the direction of the middle value \(b\).

The usefulness of \(\Theta_R\) comes from the fact that its change under a braid move is completely local.  More precisely, when a braid \(i(i+1)i\) is replaced by \((i+1)i(i+1)\), the only second-highest contributions that change are the two Demazure shortenings described in \cref{lem:key-obs}.  Their support intervals coincide, and the middle value changes from forward to backward.  This gives the following variation formula.

\begin{proposition}\label{prop:theta-braid-variation}
Let
\[
R_1=T_1\,i(i+1)i\,T_2,
\qquad
R_2=T_1\,(i+1)i(i+1)\,T_2
\]
be two reduced words for \(w_0\). Let \(u=w(T_1)\), and write
\[
u(i)=a,\qquad u(i+1)=b,\qquad u(i+2)=c.
\]
Assume that the orientation is chosen so that
\[
T_{R_1}(a,b,c)=+1,\qquad T_{R_2}(a,b,c)=-1.
\]
Then for any \((d,e,f)\in\mathcal{T}_n\),
\[
\Theta_{R_1}(d,e,f)-\Theta_{R_2}(d,e,f)=
\begin{cases}
-2, & \text{if } (a,b,c)=(d,e,f),\\
0,  & \text{if } (a,b,c)\ne(d,e,f).
\end{cases}
\]
\end{proposition}

\begin{proof}
Keep the associated shortened words
\[
R_1'=T_1\,i\,T_2,\qquad
R_2'=T_1\,(i+1)\,T_2,
\]
and the Demazure elements
\[
v_1=\Dem(R_1'),\qquad v_2=\Dem(R_2')
\]
as in \cref{lem:key-obs}.
By \cref{lem:hecke-expansion},
\[
F(R_1)-F(R_2)=\sum_{x\in S_n}\bigl(\wt_{R_1}(x)-\wt_{R_2}(x)\bigr)\Ttilde_x.
\]
On the other hand, \eqref{eq:braid-diff} gives
\[
F(R_1)-F(R_2)=q^{N-1}\bigl(\Ttilde_{v_1}-\Ttilde_{v_2}\bigr)+O(q^{N-2}).
\]
Since $\{\Ttilde_x\}_{x\in S_n}$ is a basis, comparing coefficients of $\Ttilde_x$ yields
\[
\wt_{R_1}(x)-\wt_{R_2}(x)=
\begin{cases}
q^{N-1}+O(q^{N-2}), & \text{if } x=v_1,\\
-q^{N-1}+O(q^{N-2}),& \text{if } x=v_2,\\
O(q^{N-2}), & \text{if } x\neq v_1,v_2.
\end{cases}
\]
Taking the coefficient of $q^{N-1}$ on both sides, we obtain
\begin{equation}\label{eq:D-transport}
D_{R_1}(x)-D_{R_2}(x)=
\begin{cases}
1, & \text{if } x=v_1,\\
-1,& \text{if } x=v_2,\\
0, & \text{if } x\neq v_1,v_2.
\end{cases}
\end{equation}
By \cref{def:Theta} and \eqref{eq:D-transport},
\[\begin{aligned}
    \Theta_{R_1}(d,e,f)-\Theta_{R_2}(d,e,f)
    &=\sum_{\substack{x\in S_n\\ \Supp(x)=[d,f]\cap\mathbb Z}}
    \bigl(D_{R_1}(x)-D_{R_2}(x)\bigr)\sigma_e(x)\\
    &=
    \mathbf{1}_{\Supp(v_1)=[d,f]\cap\mathbb Z}\,\sigma_e(v_1)
    -
    \mathbf{1}_{\Supp(v_2)=[d,f]\cap\mathbb Z}\,\sigma_e(v_2).
\end{aligned}\]
By \cref{lem:key-obs}, $\Supp(v_1)=\Supp(v_2)=[a,c]\cap\mathbb Z$. Hence if $(d,f)\neq(a,c)$, then both indicators vanish and the difference is $0$.

Now assume $(d,f)=(a,c)$. Then the above becomes
\[
\Theta_{R_1}(a,e,c)-\Theta_{R_2}(a,e,c)=\sigma_e(v_1)-\sigma_e(v_2).
\]
\Cref{lem:key-obs} shows that $\Fwd(v_1)$ and $\Fwd(v_2)$ differ exactly by the single element $b=u(i+1)$:
indeed $\Fwd(v_2)=\{u(j):j\le i,\ u(j)\in[a,c]\}$ while $\Fwd(v_1)=\{u(j):j\le i+1,\ u(j)\in[a,c]\}$,
so
\[
\Fwd(v_1)=\Fwd(v_2)\sqcup\{b\},\qquad \Bwd(v_2)=\Bwd(v_1)\sqcup\{b\}.
\]
Consequently, for $e\neq b$ we have $\sigma_e(v_1)=\sigma_e(v_2)$, hence $\Theta_{R_1}(a,e,c)-\Theta_{R_2}(a,e,c)=0$.
For $e=b$, we have $b\in\Fwd(v_1)$ and $b\in\Bwd(v_2)$, so $\sigma_b(v_1)=-1$ and $\sigma_b(v_2)=+1$, hence
\[
\Theta_{R_1}(a,b,c)-\Theta_{R_2}(a,b,c)=(-1)-(+1)=-2,
\]
which completes the proof.
\end{proof}

The preceding proposition affects exactly the same triple as the corresponding variation of \(T_R\) under a braid move, but with the opposite sign.

\begin{cor}\label{cor:theta-T-invariant-under-braid}
Retain the notation of the preceding proposition. Then for any
\((d,e,f)\in\mathcal{T}_n\),
\[
\Theta_{R_1}(d,e,f)-\Theta_{R_2}(d,e,f)
=
T_{R_2}(d,e,f)-T_{R_1}(d,e,f).
\]
Equivalently,
\[
\Theta_{R_1}(d,e,f)+T_{R_1}(d,e,f)
=
\Theta_{R_2}(d,e,f)+T_{R_2}(d,e,f).
\]
\end{cor}

\begin{proof}
    By \cref{prop:braid-flips-one-T} and the condition, we have
    \[
    T_{R_2}(d,e,f)-T_{R_1}(d,e,f)=
    \begin{cases}
    -2, & \text{if } (a,b,c)=(d,e,f),\\
    0,  & \text{if } (a,b,c)\ne(d,e,f),
    \end{cases}
    \]
    which is equal to $\Theta_{R_1}(d,e,f)-\Theta_{R_2}(d,e,f)$.
\end{proof}

The preceding corollary shows that the sum \(\Theta_R+T_R\) is unchanged under a single braid move. It is also unchanged under commutation moves, since both \(T_R\) and \(\Theta_R\) depend only on the commutation class of \(R\). Using the connectedness of the reduced-word graph, we obtain a quantity independent of the chosen reduced word.

\begin{cor}\label{cor:constant_sum}
For every \((d,e,f)\in \mathcal{T}_n\), the quantity
\[
\Theta_R(d,e,f)+T_R(d,e,f)
\]
is independent of \(R\in\Red(w_0)\). Hence there exists a constant \(C(d,e,f)\) such that
\[
\Theta_R(d,e,f)+T_R(d,e,f)=C(d,e,f).
\]
\end{cor}

\begin{proof}
By \cref{prop:commutation-classes-via-T}, \(T_R\) depends only on the commutation class of \(R\), and hence is preserved under commutation moves.  The statistic \(\Theta_R\) is also preserved under commutation moves, since a commutation move does not change the Hecke product \(F(R)\), and therefore does not change the coefficients \(D_R(x)\).  By \cref{cor:theta-T-invariant-under-braid}, braid moves preserve the sum \(\Theta_R+T_R\).  Since the reduced-word graph \(G(w_0)\) is connected, the quantity \(\Theta_R(d,e,f)+T_R(d,e,f)\) is independent of \(R\).
\end{proof}

At this point the precise value of \(C(d,e,f)\) is not needed: its independence of \(R\) already suffices for the proof of the main theorem. We will determine this constant in \cref{subsec:constant-C}.


\subsection{Proof of the main theorem}\label{subsec:main-theorem-proof}

We first prove the theorem in the case of reduced words for the longest element.

\begin{proposition}\label{prop:w0-main}
Let \(R_1,R_2\in\Red(w_0)\). If \([R_1]\ne [R_2]\), then
\[
\mathbb P_{R_1,q}\ne \mathbb P_{R_2,q}
\]
as rational functions of $q$.
\end{proposition}

\begin{proof}
By \cref{prop:commutation-classes-via-T}, there exists a triple
\((d,e,f)\in\mathcal{T}_n\) such that
\[
T_{R_1}(d,e,f)\ne T_{R_2}(d,e,f).
\]
On the other hand, by \cref{cor:constant_sum}, the quantity
\[
\Theta_R(d,e,f)+T_R(d,e,f)=C(d,e,f)
\]
is independent of \(R\). Hence
\[
\Theta_{R_1}(d,e,f)\ne \Theta_{R_2}(d,e,f).
\]

By the definition of \(\Theta_R\), this means that
\[
\sum_{\substack{x\in S_n\\ \Supp(x)=[d,f]\cap\mathbb Z}}
D_{R_1}(x)\sigma_e(x)
\ne
\sum_{\substack{x\in S_n\\ \Supp(x)=[d,f]\cap\mathbb Z}}
D_{R_2}(x)\sigma_e(x).
\]
Therefore there exists some \(x\in S_n\) such that
\[
D_{R_1}(x)\ne D_{R_2}(x).
\]
Since \(D_R(x)=[q^{N-1}]\wt_R(x)\), we have
\[
\wt_{R_1}(x)\ne \wt_{R_2}(x)
\]
as polynomials of $q$.
Thus \(F(R_1)\ne F(R_2)\). By \cref{cor:hecke-characterization}, this implies
\[
\mathbb P_{R_1,q}\ne \mathbb P_{R_2,q}.
\]
\end{proof}

The proposition proves the desired separation result for reduced words of the longest element $w_0$. We now return to arbitrary reduced words. The reduction in \cref{cor:reduction-w0} allows us to pass from the general case to the case of \(w_0\).

\begin{proof}[Proof of \cref{thm:main}]
If \(R_1\) and \(R_2\) are related by commutation moves, then the Hecke products
\[
\prod_{i\in R_1}(1+T_i)
\qquad\text{and}\qquad
\prod_{i\in R_2}(1+T_i)
\]
are equal. Hence \(\mathbb P_{R_1,q}=\mathbb P_{R_2,q}\) by \cref{cor:hecke-characterization}.

Conversely, suppose \(\mathbb P_{R_1,q}=\mathbb P_{R_2,q}\). By \cref{cor:reduction-w0}, it suffices to consider the case \(R_1,R_2\in\Red(w_0)\). In this case, \cref{prop:w0-main} implies that \([R_1]=[R_2]\). Hence \(R_1\) and \(R_2\) are related by commutation moves.
\end{proof}

This completes the proof of the main theorem. We next determine the constant \(C(d,e,f)\) appearing in \cref{cor:constant_sum}.

\subsection{The vanishing of the constant \texorpdfstring{\(C\)}{C}}
\label{subsec:constant-C}

Since the constant \(C(d,e,f)\) is independent of the chosen reduced word, it suffices to compute it for one convenient reduced word of \(w_0\). Recall from \cref{subsec:commutation-classes} the standard reduced word
\[
R_0=1(21)(321)\cdots((n-1)(n-2)\cdots1).
\]
Equivalently, write
\[
R_0=B_2B_3\cdots B_n,
\qquad
B_c=(c-1)(c-2)\cdots1.
\]

We first record the one-letter Demazure shortenings of \(R_0\).

\begin{lemma}\label{lem:R0-shortenings}
Let \(1\le a<c\le n\). Let \(R_0^{(a,c)}\) be the word obtained from \(R_0\) by deleting the letter \(c-a\) in the block \(B_c\). Equivalently, this is the letter whose application creates the inversion \((a,c)\) in the standard word \(R_0\). Define
\[
v_{a,c}:=\Dem(R_0^{(a,c)}).
\]
Then
\[
v_{a,c}
=
n(n-1)\cdots(c+1)\,(c-1)(c-2)\cdots a\,c\,(a-1)\cdots 1.
\]
Equivalently,
\[
w_0=(a\ a+1\ \cdots\ c)\,v_{a,c}.
\]
Consequently,
\[
\Supp(v_{a,c})=[a,c]\cap\mathbb Z,
\]
\[
\Fwd(v_{a,c})=\{a,a+1,\ldots,c-1\},
\qquad
\Bwd(v_{a,c})=\{c\}.
\]
\end{lemma}

\begin{proof}
After applying the blocks \(B_2,\ldots,B_{c-1}\), the current permutation is
\[
(c-1)(c-2)\cdots 1\,c\,(c+1)\cdots n.
\]
In the full block \(B_c=(c-1)(c-2)\cdots 1\), the value \(c\) moves leftward across \(1,2,\ldots,c-1\). The letter \(c-a\) is precisely the step at which \(c\) would cross \(a\).

If this letter is deleted, then after the first \(a-1\) steps of \(B_c\), the value \(c\) remains immediately to the right of \(a\). The remaining letters in \(B_c\) would act on adjacent decreasing pairs, hence they are skipped in the Demazure product. Thus after the shortened block \(B_c\), the relative order of the values \(1,\ldots,c\) is
\[
(c-1)(c-2)\cdots a\,c\,(a-1)\cdots 1.
\]
The later blocks \(B_{c+1},\ldots,B_n\) move the larger values \(c+1,\ldots,n\) successively to the front. Each such move is length-increasing and therefore is taken by the Demazure product. Hence
\[
v_{a,c}
=
n(n-1)\cdots(c+1)\,(c-1)(c-2)\cdots a\,c\,(a-1)\cdots 1.
\]

Comparing this one-line notation with
\[
w_0=n(n-1)\cdots(c+1)\,c(c-1)\cdots a\,(a-1)\cdots 1
\]
shows that applying the cycle \((a\ a+1\ \cdots\ c)\) to the values of \(v_{a,c}\) gives \(w_0\). Thus
\[
w_0=(a\ a+1\ \cdots\ c)\,v_{a,c}.
\]

The displayed formula also immediately gives the position comparison with \(w_0\). The values outside \([a,c]\) occupy the same positions as in \(w_0\). For \(a\le t<c\), the value \(t\) is shifted one position to the left relative to \(w_0\), so \(t\in\Fwd(v_{a,c})\). The value \(c\) is shifted to the position occupied by \(a\) in \(w_0\), so \(c\in\Bwd(v_{a,c})\). Hence
\[
\Supp(v_{a,c})=[a,c]\cap\mathbb Z,
\qquad
\Fwd(v_{a,c})=\{a,a+1,\ldots,c-1\},
\qquad
\Bwd(v_{a,c})=\{c\}.
\]
\end{proof}

We now compute \(T_{R_0}\) and \(\Theta_{R_0}\).

\begin{proposition}\label{prop:R0-computation}
For every \((d,e,f)\in\mathcal{T}_n\), one has
\[
T_{R_0}(d,e,f)=1,
\qquad
\Theta_{R_0}(d,e,f)=-1.
\]
\end{proposition}

\begin{proof}
First consider \(T_{R_0}(d,e,f)\). In the standard word \(R_0=B_2B_3\cdots B_n\), the block \(B_c\) creates the inversions
\[
(1,c),(2,c),\ldots,(c-1,c)
\]
in this order. Therefore the inversion \((d,e)\) is created in the block \(B_e\), while the inversion \((e,f)\) is created in the later block \(B_f\). Since \(e<f\), we have
\[
\tau_{R_0}(d,e)<\tau_{R_0}(e,f),
\]
and hence
\[
T_{R_0}(d,e,f)=1.
\]

We next compute \(\Theta_{R_0}(d,e,f)\). By the combinatorial interpretation of \(D_R(x)\), the coefficient \(D_{R_0}(x)\) counts one-letter Demazure shortenings of \(R_0\) with Demazure product \(x\). By \cref{lem:R0-shortenings}, deleting the letter corresponding to the inversion \((a,c)\) gives a Demazure product \(v_{a,c}\) with
\[
\Supp(v_{a,c})=[a,c]\cap\mathbb Z.
\]
Thus, in the definition
\[
\Theta_{R_0}(d,e,f)
=
\sum_{\substack{x\in S_n\\ \Supp(x)=[d,f]\cap\mathbb Z}}
D_{R_0}(x)\sigma_e(x),
\]
the only one-letter shortening which contributes is the deletion corresponding to \((d,f)\), namely \(v_{d,f}\). Therefore
\[
\Theta_{R_0}(d,e,f)=\sigma_e(v_{d,f}).
\]
Again by \cref{lem:R0-shortenings},
\[
\Fwd(v_{d,f})=\{d,d+1,\ldots,f-1\}.
\]
Since \(d<e<f\), we have \(e\in\Fwd(v_{d,f})\). Hence
\[
\sigma_e(v_{d,f})=-1,
\]
and therefore
\[
\Theta_{R_0}(d,e,f)=-1.
\]
\end{proof}

\begin{cor}\label{cor:C-vanishes}
For every \((d,e,f)\in\mathcal{T}_n\), one has
\[
C(d,e,f)=0.
\]
Equivalently, for every \(R\in\Red(w_0)\),
\[
\Theta_R(d,e,f)=-T_R(d,e,f).
\]
\end{cor}

\begin{proof}
By \cref{cor:constant_sum}, the quantity
\[
\Theta_R(d,e,f)+T_R(d,e,f)
\]
is independent of \(R\). Hence we may compute it at \(R=R_0\). By \cref{prop:R0-computation},
\[
C(d,e,f)
=
\Theta_{R_0}(d,e,f)+T_{R_0}(d,e,f)
=
(-1)+1
=
0.
\]
The equivalent statement follows immediately from \cref{cor:constant_sum}.
\end{proof}

Thus the constant introduced in \cref{cor:constant_sum} vanishes. In particular, the signed second-highest statistic \(\Theta_R\) recovers the commutation-class invariant \(T_R\) with the opposite sign.

\section{Hecke-theoretic remarks and counterexamples}
\label{sec:further-remarks}

This section records two supplementary observations.  First, we keep the
cell-representation argument which distinguishes a useful class of local braid
moves.  Second, we explain why the Temperley--Lieb quotient is too coarse to
replace the full Hecke-algebra element appearing in \Cref{cor:hecke-characterization}.

\subsection{Partial results in the Hecke algebra}
Consider the right cell $W^{(R)}(s_1)$ containing $C_{s_1}$.  By
\Cref{prop:Tw-basis} and \Cref{corr:5.4.2}, one has
$C_w\in W^{(R)}(s_1)$ if and only if $w\equiv_K s_1$.  A direct calculation gives
\[
W^{(R)}(s_1)=\{s_1,\ s_1s_2,\ s_1s_2s_3,\ldots,\ s_1s_2\cdots s_{n-1}\}.
\]
For $1\le r\le n-1$, set
\[
    w_r:=s_1s_2\cdots s_r.
\]
Thus the ordered basis of $W^{(R)}(s_1)$ is
\[
    \{C_{w_1},C_{w_2},\ldots,C_{w_{n-1}}\}.
\]

\begin{proposition}\label{prop:right-cell-matrices-compact}
Let \(A_j\) denote the matrix of the right multiplication operator
\[
    m_j: v\longmapsto v(1+T_j)
\]
on \(W^{(R)}(s_1)\), with respect to the ordered basis
\[
    \{C_{w_1},C_{w_2},\ldots,C_{w_{n-1}}\}.
\]
We use the convention that columns record the images of basis vectors.  Then the
matrices \(A_j=[1+T_j]_{W^{(R)}(s_1)}\) have the following block forms.

First,
{\renewcommand{\arraystretch}{1.25}\setlength{\arraycolsep}{10pt}
\[
A_1=
\begin{pmatrix}
\begin{matrix}
0 & q^{\frac12}\\
0 & q+1
\end{matrix}
& \mathbf{0}_{2\times(n-3)}\\[6pt]
\mathbf{0}_{(n-3)\times 2} & (q+1)I_{n-3}
\end{pmatrix}.
\]
}

For \(2\le j\le n-2\), write
\[
B_j:=
\begin{pmatrix}
q+1 & 0 & 0\\
q^{\frac12} & 0 & q^{\frac12}\\
0 & 0 & q+1
\end{pmatrix}.
\]
Then
{\renewcommand{\arraystretch}{1.25}\setlength{\arraycolsep}{10pt}
\[
A_j=
\begin{pmatrix}
(q+1)I_{j-2} & \mathbf{0}_{(j-2)\times 3} & \mathbf{0}_{(j-2)\times(n-j-2)}\\[6pt]
\mathbf{0}_{3\times(j-2)} & B_j & \mathbf{0}_{3\times(n-j-2)}\\[6pt]
\mathbf{0}_{(n-j-2)\times(j-2)} & \mathbf{0}_{(n-j-2)\times 3} & (q+1)I_{n-j-2}
\end{pmatrix}.
\]
}

Finally,
{\renewcommand{\arraystretch}{1.25}\setlength{\arraycolsep}{10pt}
\[
A_{n-1}=
\begin{pmatrix}
(q+1)I_{n-3} & \mathbf{0}_{(n-3)\times 2}\\[6pt]
\mathbf{0}_{2\times(n-3)} &
\begin{matrix}
q+1 & 0\\
q^{\frac12} & 0
\end{matrix}
\end{pmatrix}.
\]
}
\end{proposition}

\begin{proof}
We compute the action of right multiplication by \(1+T_j\) on the basis vectors
\(C_{w_r}\) and then read off the matrices.  The computation uses the
multiplication formula in \Cref{prop:KL-right-mult}, working modulo the lower
right-cell terms \(H(<_R s_1)\).

First, if \(r=j\), then \(w_rs_j<w_r\).  Hence
\[
    C_{w_r}(1+T_j)=0
\]
in the quotient \(W^{(R)}(s_1)\).

Next suppose \(r=j-1\).  Then \(w_rs_j=w_{r+1}\), and the multiplication formula gives
\[
    C_{w_r}(1+T_j)
    \equiv (q+1)C_{w_r}+q^{\frac12}C_{w_{r+1}}
    \pmod{H(<_R s_1)}.
\]
Similarly, if \(r=j+1\), then
\[
    C_{w_r}(1+T_j)
    \equiv q^{\frac12}C_{w_{r-1}}+(q+1)C_{w_r}
    \pmod{H(<_R s_1)}.
\]
Finally, if \(|r-j|>1\), then all terms produced by \Cref{prop:KL-right-mult},
except the diagonal term \(C_{w_r}\), lie in \(H(<_R s_1)\).  Therefore
\[
    C_{w_r}(1+T_j)
    \equiv (q+1)C_{w_r}
    \pmod{H(<_R s_1)}.
\]
Thus, in \(W^{(R)}(s_1)\), we have
\[
C_{w_r}(1+T_j)=
\begin{cases}
0,& r=j,\\[4pt]
(q+1)C_{w_r}+q^{\frac12}C_{w_{r+1}},& r=j-1,\\[4pt]
q^{\frac12}C_{w_{r-1}}+(q+1)C_{w_r},& r=j+1,\\[4pt]
(q+1)C_{w_r},& |r-j|>1.
\end{cases}
\]
The three displayed matrix forms follow immediately from this formula, using the
column convention stated above.
\end{proof}

For a word \(R=i_1i_2\cdots i_m\), let
\[
    M(R):=
    A_{i_m}A_{i_{m-1}}\cdots A_{i_1}.
\]
Thus \(M(R)\) is the matrix of right multiplication by
\[
    (1+T_{i_1})(1+T_{i_2})\cdots(1+T_{i_m})
\]
on the right-cell representation \(W^{(R)}(s_1)\).  All kernels below are taken
over the field \(\mathbb Q(q^{\pm  1/2})\).

Define \(K(R)\subseteq\{1,2,\ldots,n-1\}\) to be the set of indices \(t\) such that,
using only commutation moves, the word \(R\) can be transformed into a word whose
first letter is \(t\).

\begin{lemma}\label{lem:kernel-K}
For every word \(R=i_1i_2\cdots i_m\), one has
\[
\ker M(R)\cap \{C_{w_t}:1\le t\le n-1\}
=
\{C_{w_t}:t\in K(R)\}.
\]
Here a basis vector \(C_{w_t}\) is identified with its coordinate vector in the cell
representation \(W^{(R)}(s_1)\).
\end{lemma}

\begin{proof}
We prove the statement by induction on \(m\).  When \(m=1\), say \(R=i_1\),
\Cref{prop:right-cell-matrices-compact} gives
\[
    \ker A_{i_1}
    =
    \operatorname{Span}\{C_{w_{i_1}}\}.
\]
Since \(K(R)=\{i_1\}\), the claim follows.

Assume the statement holds for words of length \(m-1\), and write
\[
    R=i_1R',
    \qquad R'=i_2i_3\cdots i_m.
\]
From the definition of \(K(R)\), we have
\[
    K(R)=\{i_1\}\cup\bigl(K(R')\setminus\{i_1-1,i_1+1\}\bigr),
\]
where indices outside \(\{1,\ldots,n-1\}\) are ignored.

Let \(t\notin\{i_1-1,i_1,i_1+1\}\).  By
\Cref{prop:right-cell-matrices-compact},
\[
    A_{i_1}C_{w_t}=(q+1)C_{w_t}.
\]
Therefore
\[
\begin{aligned}
C_{w_t}\in\ker M(R)
&\Longleftrightarrow
M(R')A_{i_1}C_{w_t}=0\\
&\Longleftrightarrow
M(R')\bigl((q+1)C_{w_t}\bigr)=0\\
&\Longleftrightarrow
C_{w_t}\in\ker M(R')\\
&\Longleftrightarrow
t\in K(R'),
\end{aligned}
\]
where we used that \(q+1\) is nonzero, hence invertible, in
\(\mathbb Q(q^{1/2})\).  This proves the desired equivalence for all such \(t\).

Also,
\[
    A_{i_1}C_{w_{i_1}}=0,
\]
so \(C_{w_{i_1}}\in\ker M(R)\), matching the fact that \(i_1\in K(R)\).

It remains to show that, whenever the indices exist,
\[
    C_{w_{i_1-1}},\ C_{w_{i_1+1}}\notin\ker M(R).
\]
We prove the assertion for \(C_{w_{i_1+1}}\); the other case is identical.  By
\Cref{prop:right-cell-matrices-compact},
\[
A_{i_1}C_{w_{i_1+1}}
=
q^{\frac12}C_{w_{i_1}}+(q+1)C_{w_{i_1+1}}.
\]
Thus after the first multiplication, the coefficients of the two adjacent basis
vectors \(C_{w_{i_1}}\) and \(C_{w_{i_1+1}}\) are both nonzero polynomials with
nonnegative integer coefficients in \(q^{1/2}\).

We use the following elementary observation.  Suppose
\[
    v=\sum_{r=1}^{n-1} f_r(q^{1/2})C_{w_r},
    \qquad f_r\in\mathbb Z_{\ge0}[q^{1/2}],
\]
and suppose that, for some \(a\), both \(f_a\) and \(f_{a+1}\) are nonzero.  Then,
for any \(j\), the vector
\[
    A_jv
    =
    \sum_{r=1}^{n-1} g_r(q^{1/2})C_{w_r}
\]
has all \(g_r\in\mathbb Z_{\ge0}[q^{1/2}]\), and \(g_a,g_{a+1}\) are still both
nonzero.  Indeed, this follows directly from the four cases in
\Cref{prop:right-cell-matrices-compact}.  If \(j\notin\{a,a+1\}\), then \(g_a\)
contains the nonzero contribution \((q+1)f_a\), and \(g_{a+1}\) contains the
nonzero contribution \((q+1)f_{a+1}\).  If \(j=a\), then \(g_a\) receives the
nonzero contribution \(q^{1/2}f_{a+1}\), and \(g_{a+1}\) receives the nonzero
contribution \((q+1)f_{a+1}\).  If \(j=a+1\), then \(g_a\) receives the nonzero
contribution \((q+1)f_a\), and \(g_{a+1}\) receives the nonzero contribution
\(q^{1/2}f_a\).  Since all coefficients are nonnegative, no cancellation is
possible.

Applying this observation successively to the remaining factors
\(A_{i_2},\ldots,A_{i_m}\), we see that the resulting vector is nonzero.  Hence
\(C_{w_{i_1+1}}\notin\ker M(R)\).  The same argument with the adjacent pair
\((i_1-1,i_1)\) proves \(C_{w_{i_1-1}}\notin\ker M(R)\).  The induction is
complete.
\end{proof}

\begin{theorem}\label{thm:short-local-braid-separation}
Let \(R_1\) and \(R_2\) be reduced words.  Suppose that, for some \(i\) and some
word \(Q\) of length \(0\) or \(1\), the words begin as
\[
    R_1=Q\,i\,(i+1)\,i\,R,
    \qquad
    R_2=Q\,(i+1)\,i\,(i+1)\,R
\]
for the same remaining word \(R\).  Then
\[
    \mathbb P_{R_1,q}\ne\mathbb P_{R_2,q}.
\]
\end{theorem}

\begin{proof}
By \Cref{cor:hecke-characterization}, it is enough to prove that the two Hecke
products
\[
    F(R_1)
    \qquad\text{and}\qquad
    F(R_2)
\]
are different.  It is therefore enough to show that their images in the right-cell
representation \(W^{(R)}(s_1)\) are different.

By \Cref{lem:kernel-K}, the intersection of the kernel of the corresponding
right-cell operator with the basis vectors \(\{C_{w_t}:1\le t\le n-1\}\) records
exactly the set \(K(\cdot)\).

If \(Q\) is empty, then \(i\in K(R_1)\), but \(i\notin K(R_2)\), since in \(R_2\)
the letter \(i\) is blocked by the initial letter \(i+1\).  Hence
\(K(R_1)\ne K(R_2)\).

Now suppose \(Q\) has length one, say \(Q=j\).  Since \(R_1\) and \(R_2\) are both
reduced, \(j\notin\{i,i+1\}\).  If \(j\) commutes with both \(i\) and \(i+1\), then
the same argument as above gives \(i\in K(R_1)\) and \(i\notin K(R_2)\).  If
\(j=i-1\), then \(i+1\in K(R_2)\), because \(i+1\) commutes with \(i-1\), whereas
\(i+1\notin K(R_1)\), because it is blocked by the preceding letter \(i\).  If
\(j=i+2\), the symmetric argument gives \(i\in K(R_1)\), but \(i\notin K(R_2)\).
These are all possible noncommuting cases for a one-letter \(Q\).

Thus \(K(R_1)\ne K(R_2)\) in every case.  By \Cref{lem:kernel-K}, the two
right-cell operators have different kernels on the displayed basis vectors, so
they are different operators.  Consequently the two Hecke products are different.
Using \Cref{cor:hecke-characterization}, we obtain
\[
    \mathbb P_{R_1,q}\ne\mathbb P_{R_2,q}.
\]
\end{proof}

\subsection{Counterexample in the Temperley--Lieb algebra}
We now turn to a smaller quotient of the Hecke algebra.  Let $A=\mathbb Z[q^{\pm \frac{1}{2}}]$.
The Temperley--Lieb algebra $\mathrm{TL}_n(q+1)$ over $A$ is generated by
\[
    U_1,U_2,\ldots,U_{n-1}
\]
subject to
\[
\begin{aligned}
U_i^2&=(q+1)U_i, &&1\le i\le n-1,\\
U_iU_{i\pm1}U_i&=U_i, &&1\le i\le n-2,\\
U_iU_j&=U_jU_i, &&|i-j|>1.
\end{aligned}
\]
The Temperley--Lieb algebra plays an important role in low-dimensional topology and mathematical physics, particularly through its connections with knot invariants, statistical mechanics, and quantum spin systems.

We note that there is a natural quotient map
\[
    \phi:H_n(q)\twoheadrightarrow \mathrm{TL}_n(q+1),
    \qquad
    \phi(T_i)=U_i-1,
\]
or equivalently
\[
    \phi(1+T_i)=U_i.
\]
Thus, for a word $R=i_1i_2\cdots i_L$, the image of the Hecke product
\[
    \Phi_R:=\prod_{r=1}^{L}(1+T_{i_r})
\]
is simply
\[
    \phi(\Phi_R)=U_{i_1}U_{i_2}\cdots U_{i_L}.
\]

The following example shows that this quotient loses information which is
relevant to the distribution $\mathbb P_{R,q}$.

\begin{ex}\label{ex:TL-counterexample}
In $S_8$, consider the two reduced words
\[
\begin{aligned}
R_1={}&1\,3\,4\,3\,2\,3\,5\,4\,5\,3\,1\,2\,1\,6\,7\,5\,6\,5\,4\,3\,4\,2,\\
R_2={}&1\,3\,4\,3\,2\,3\,5\,4\,5\,3\,2\,1\,2\,6\,7\,5\,6\,5\,4\,3\,4\,2.
\end{aligned}
\]
They differ by the braid move $121\leftrightarrow 212$ in the middle.  Hence
they are not related by commutation moves alone.  However, in
$\mathrm{TL}_8(q+1)$ one has
\[
\begin{aligned}
&U_1U_3U_4U_3U_2U_3U_5U_4U_5U_3U_1U_2U_1U_6U_7U_5U_6U_5U_4U_3U_4U_2\\
={}&U_1U_3U_4U_3U_2U_3U_5U_4U_5U_3U_2U_1U_2U_6U_7U_5U_6U_5U_4U_3U_4U_2.
\end{aligned}
\]
Equivalently,
\[
    \phi(\Phi_{R_1})=\phi(\Phi_{R_2}).
\]
Thus the Temperley--Lieb quotient cannot distinguish these two reduced words,
although the full Hecke-algebra product does distinguish commutation classes by
\Cref{thm:main} and \Cref{cor:hecke-characterization}.
\end{ex}

\bigskip

\section{Further directions}\label{sec:further-directions}

Our main theorem shows that the family of distributions
\[
\{\mathbb P_{R,q}:q\in \mathbb Z^+\}
\]
remembers exactly the commutation class of a reduced word \(R\). In other words, if two reduced words \(R_1\) and \(R_2\) induce the same distribution for every positive integer \(q\), then they must be related by commutation moves only.

It is natural to ask whether the same conclusion already holds after specializing \(q\) to a single value. A particularly interesting specialization is \(q=1\). Under this specialization, the Hecke algebra \(H_n(q)\) degenerates to the group algebra \(\mathbb Z[S_n]\). Indeed, in the usual normalization, the quadratic relation
\[
T_i^2=(q-1)T_i+q
\]
becomes
\[
T_i^2=1
\]
when \(q=1\). Thus the Hecke-algebraic expression controlling \(\mathbb P_{R,q}\) specializes to an expression inside the ordinary group algebra of \(S_n\).

This suggests the following question. Although the full one-parameter family \(\mathbb P_{R,q}\) distinguishes commutation classes, it is not clear whether a single specialization \(\mathbb P_{R,q_0}\) should still contain enough information to do so.

\begin{quest}\label{quest:fixed-q}
Does there exist a positive integer \(q\), and two reduced words \(R_1,R_2\) of the same permutation, such that \(R_1\) and \(R_2\) are not in the same commutation class but
\[
\mathbb P_{R_1,q}=\mathbb P_{R_2,q}?
\]
Equivalently, does equality of the distributions at a fixed value of \(q\) force equality of commutation classes?
\end{quest}

The case \(q=1\) is especially natural, since it corresponds to the degeneration from the Hecke algebra to the group algebra. Understanding this specialization may clarify how much of the commutation-class information is genuinely Hecke-theoretic, and how much is already visible in the ordinary group algebra \(\mathbb Z[S_n]\).

\section*{Acknowledgements}

This work was carried out during the PACE Undergraduate Research Program in Algebraic Combinatorics at Peking University. We are deeply grateful to Shiliang Gao for his guidance, insightful discussions, and helpful observations throughout the project. We also thank the organizers of PACE for providing this research opportunity.

\appendix

\section{Hecke algebra and cell representations}\label{app:hecke}

Let \(R=\mathbb Z[v,v^{-1}]\), where \(v=q^{1/2}\) is a formal parameter with \(q\) an indeterminate, and let
\[
\mathcal{S} = \{ s_i \mid 1 \le i < n \}
\]
denote the simple transpositions in the symmetric group $S_n$.
The Hecke algebra $H_n(q)$ of type~$A_{n-1}$ is the unital associative $R$--algebra generated by elements $T_i$ $(1 \le i < n)$ subject to the braid relations and the quadratic relation:
\begin{align}
T_i T_j &= T_j T_i, && \text{if } |i-j| \ge 2, \label{eq:braid-comm}\\
T_i T_{i+1} T_i &= T_{i+1} T_i T_{i+1}, \label{eq:braid}\\
T_i^2 &= (q - 1) T_i + q. \label{eq:quadratic}
\end{align}
In particular, when $q = 1$ the algebra $H_n(q)$ specializes to the group algebra $\mathbb{Z}[S_n]$.

For \(w\in S_n\), choose any reduced expression $w = s_{i_1} \cdots s_{i_\ell}$ and set
\[
T_w := T_{i_1} \cdots T_{i_\ell},
\]
which is well-defined by the braid relations \eqref{eq:braid-comm}--\eqref{eq:braid}.

\begin{proposition}\label{prop:Tw-basis}
The elements $\{ T_w \}_{w \in {S}_n}$ form an $A$--basis of $H_n(q)$.
\end{proposition}

Right multiplication by the generators satisfies the multiplication rule
\begin{equation}\label{eq:TwTj}
T_w T_{s_j} =
\begin{cases}
T_{ws_j}, & \text{if } ws_j > w \text{ in the Bruhat order},\\[4pt]
(q - 1) T_w + q\,T_{ws_j}, & \text{if } ws_j < w,
\end{cases}
\end{equation}
and an analogous identity holds for left multiplication.

There are two elementary one-dimensional representations (the $q$--analogues of the trivial and sign characters): for all $w\in{S}_n$,
\[
\rho_{\mathrm{triv}}(T_w)=q^{\ell(w)},\qquad \rho_{\mathrm{sgn}}(T_w)=(-1)^{\ell(w)}.
\]

\subsection*{The Kazhdan--Lusztig basis}
Let $\overline{\cdot}\colon A\to A$ be the ring involution given by $q^{\frac{1}{2}}\mapsto q^{-\frac{1}{2}}$, and extend it semilinearly to an involution $\iota$ on $H_n(q)$. There is also an anti-involution $\star$ determined by $T_{s_i}^\star=T_{s_i}$ (equivalently $T_w^\star=T_{w^{-1}}$).

Introduce the renormalized basis $\widehat T_w := q^{-\ell(w)/2} T_w$.
The Kazhdan--Lusztig theorem asserts the existence and uniqueness of a basis $\{C_w\}_{w\in{S}_n}$ of $H_n(q)$ characterised by:
\begin{itemize}
  \item $\iota(C_w)=C_w$ (self-duality under the bar-involution);
  \item $C_w \in \widehat T_w + \sum_{y<w} q^{1/2}\mathbb{Z}[q^{1/2}]\,\widehat T_y$,
\end{itemize}
where $<$ denotes Bruhat order. Equivalently,
\[
C_w=\widehat T_w+\sum_{y<w} a_{y,w}(q)\,\widehat T_y,\qquad a_{y,w}(q)\in q^{1/2}\mathbb{Z}[q^{1/2}],
\]
and the coefficients can be packaged as \emph{Kazhdan--Lusztig polynomials} $P_{y,w}(q)\in\mathbb{Z}[q]$ with $P_{w,w}=1$ and $\deg P_{y,w}\le\frac{1}{2}(\ell(w)-\ell(y)-1)$ for $y<w$.
See \cite{williamson2003mind} for the construction.

\begin{defin}\label{def:mu}
For $a,b\in S_n$ with $a\le b$, let $\mu(a,b)$ denote the coefficient of the top-degree term
in $P_{a,b}(q)$, i.e.
\[
P_{a,b}(q)=\mu(a,b)\,q^{\frac{1}{2}(\ell(b)-\ell(a)-1)}+\text{(lower-degree terms)}.
\]
By convention $P_{a,b}=0$ (and $\mu(a,b)=0$) if $a\not\le b$.
We write $a \prec b$ if $a < b$ and $\mu(a,b) \neq 0$.
\end{defin}

\begin{proposition}[Multiplication formulae {\cite{williamson2003mind}}]\label{prop:KL-right-mult}
Let $s \in \mathcal{S}$ be a simple reflection and $w \in S_n$. Then
\[
C_w \, T_s =
\left\{
\begin{array}{ll}
  \rule{0pt}{4ex} 
  -\,q^{-\frac{1}{2}} \, C_w,
    & \text{if } ws < w, \\[6pt]
  \rule{0pt}{6ex} 
  q^{\frac{1}{2}} \, C_w + C_{ws}
  + \displaystyle\sum_{y \in S_n,\, y \prec ws,\, ys < y}
    \mu(y, ws) \, C_y,
    & \text{if } ws > w.
\end{array}
\right.
\]
An analogous formula holds for left multiplication:
\[
T_s \, C_w =
\left\{
\begin{array}{ll}
  \rule{0pt}{4ex}
  -\,q^{-\frac{1}{2}} \, C_w,
    & \text{if } sw < w, \\[6pt]
  \rule{0pt}{6ex}
  q^{\frac{1}{2}} \, C_w + C_{sw}
  + \displaystyle\sum_{y \in S_n,\, y \prec sw,\, sy < y}
    \mu(y, sw) \, C_y,
    & \text{if } sw > w.
\end{array}
\right.
\]
\end{proposition}

\begin{proposition}[{\cite[Proposition~4.3.3]{williamson2003mind}}]
\label{prop:4.3.3}
Let $w\in{S}_n$ and let $r\in \mathcal{S}$ be a simple reflection with $wr>w$.
Then the only element $x\in{S}_n$ satisfying $w\prec x$ and $xr<x$ is $x=wr$,
and in this case $\mu(w,x)=1$ holds.
\end{proposition}

\subsection*{Cells and the Robinson--Schensted correspondence}
Fix the Kazhdan--Lusztig basis $\{C_w\}$. Define a preorder $\leftarrow_L$ on $S_n$ by: $x\leftarrow_L y$ iff there exists $a\in H_n(q)$ such that the expansion of $aC_y$ contains $C_x$ with nonzero coefficient. Let $\le_L$ be the transitive closure of $\leftarrow_L$. Right and two-sided preorders $\le_R$, $\le_{LR}$ are defined analogously. A \emph{left cell} is an equivalence class for the relation $x\sim_L y \iff x\le_L y\le_L x$; right cells $\sim_R$ and two-sided cells $\sim_{LR}$ are defined similarly.

For $w\in{S}_n$, let $H(\le_L w)$ (resp.\ $H(<_L w)$) be the $A$--span of $\{C_x : x\le_L w\}$ (resp.\ $\{C_x : x<_L w\}$). The quotient
\[
W^{(L)}(w)\;:=\; H(\le_L w)\big/H(<_L w)
\]
is a free $A$--module whose basis consists of the images of $C_x$ with $x\sim_L w$, and it carries a natural left $H_n(q)$--module structure called the \emph{cell representation} associated to the left cell of $w$. Right cell representations $W^{(R)}(w)$ are defined analogously.

In type $A$, cells admit a concrete combinatorial description via the Robinson--Schensted correspondence.

\begin{defin}[Robinson--Schensted insertion]
Let $w=w_1w_2\cdots w_n\in{S}_n$ be a permutation considered as a word.  Define sequences of tableaux \(P(k)\) and \(Q(k)\) recursively as follows. Set \(P(0)=Q(0)=\emptyset\). For \(k=1,\dots,n\), set
\[
P(k)=P(k-1)\leftarrow w_k,\qquad
\]
\[
Q(k)=Q(k-1)\,\text{with the entry }k\text{ placed in the box of }P(k)\setminus P(k-1),
\]
where $T\leftarrow x$ denotes the usual row-insertion of the letter $x$ into the tableau $T$ (see the following example).
Then the \emph{insertion tableau} (or $P$-tableau) and the \emph{recording tableau} (or $Q$-tableau)
of $w$ are defined by $P(w):=P(n)$ and $Q(w):=Q(n)$,
so that the Robinson--Schensted correspondence associates the \(P\)-\(Q\) pair to $w$ under the map $w\mapsto (P(w),Q(w))$.
\end{defin}

\begin{ex}\label{ex:RS-31524}
Let \(w=31524\). The successive insertion and recording tableaux are as follows:
\[
\setlength{\arraycolsep}{0.8em}
\renewcommand{\arraystretch}{1.55}
\begin{array}{@{}c l c c@{}}
\toprule
k & \text{operation} & P(k) & Q(k) \\
\midrule
1
& \text{insert }3
& \RSyt{3}
& \RSyt{1}
\\[4pt]

2
& \text{insert }1;\ 3\text{ is bumped to the second row}
& \RSyt{1,3}
& \RSyt{1,2}
\\[4pt]

3
& \text{insert }5;\ \text{append to the first row}
& \RSyt{15,3}
& \RSyt{13,2}
\\[4pt]

4
& \text{insert }2;\ 2\text{ bumps }5,\text{ which is appended to the second row}
& \RSyt{12,35}
& \RSyt{13,24}
\\[4pt]

5
& \text{insert }4;\ \text{append to the first row}
& \RSyt{124,35}
& \RSyt{135,24}
\\
\bottomrule
\end{array}
\]
Therefore
\[
P(31524)=\ytableaushort{124,35},
\qquad
Q(31524)=\ytableaushort{135,24}.
\]
\end{ex}

\begin{defin}[Elementary Knuth transformations]
Let $u$ be a permutation and suppose that three adjacent letters of $u$ form a triple $\{x,y,z\}$ with $x<y<z$.
The \emph{elementary Knuth transformations} are the local rewrites
\[
\cdots zxy\cdots \longleftrightarrow \cdots xzy\cdots,\qquad
\cdots yxz\cdots \longleftrightarrow \cdots yzx\cdots,
\]
leaving the remainder of the word unchanged.  Two permutations $u$ and $v$ are called \emph{Knuth equivalent},
written $u \equiv_K  v$, if one can be obtained from the other by a finite sequence of elementary Knuth transformations.
\end{defin}

\begin{proposition}
[{\cite[Proposition~2.6.1]{williamson2003mind}}]\label{prop:2.6.1}
    Two permutations $u$ and $v$ share the same $P$-symbol if and only if they are Knuth equivalent.
\end{proposition}

\begin{proposition}[{\cite[Corrollary~5.4.2]{williamson2003mind}}]\label{corr:5.4.2}
Let $x,y\in{S}_n$.  Then $x$ and $y$ lie in the same right cell with respect to the
Kazhdan--Lusztig basis if and only if their insertion tableaux coincide:
\[
x\sim_{R} y \quad\Longleftrightarrow\quad P(x)=P(y).
\]
Equivalently, right cells are exactly the fibres of the map $w\mapsto P(w)$ under the
Robinson--Schensted correspondence.
\end{proposition}

\bibliographystyle{alpha}
\bibliography{ref} 

@misc{MPPY2025,
      title={Grothendieck Shenanigans: Permutons from pipe dreams via integrable probability}, 
      author={Alejandro H. Morales and Greta Panova and Leonid Petrov and Damir Yeliussizov},
      year={2025},
      eprint={2407.21653},
      archivePrefix={arXiv},
      primaryClass={math.PR},
      url={https://arxiv.org/abs/2407.21653}, 
}

@article {BB93,
    AUTHOR = {Bergeron, Nantel and Billey, Sara},
     TITLE = {R{C}-graphs and {S}chubert polynomials},
   JOURNAL = {Experiment. Math.},
  FJOURNAL = {Experimental Mathematics},
    VOLUME = {2},
      YEAR = {1993},
    NUMBER = {4},
     PAGES = {257--269},
      ISSN = {1058-6458,1944-950X},
   MRCLASS = {05E99 (05E05 14M15 20C30)},
  MRNUMBER = {1281474},
MRREVIEWER = {Axel\ Kohnert},
       URL = {http://projecteuclid.org/euclid.em/1048516036},
}

@article {LS82,
    AUTHOR = {Lascoux, Alain and Sch\"utzenberger, Marcel-Paul},
     TITLE = {Polyn\^omes de {S}chubert},
   JOURNAL = {C. R. Acad. Sci. Paris S\'er. I Math.},
  FJOURNAL = {Comptes Rendus des S\'eances de l'Acad\'emie des Sciences.
              S\'erie I. Math\'ematique},
    VOLUME = {294},
      YEAR = {1982},
    NUMBER = {13},
     PAGES = {447--450},
      ISSN = {0249-6291},
   MRCLASS = {14M17 (05A10 14N10)},
  MRNUMBER = {660739},
}

@article {KM05,
    AUTHOR = {Knutson, Allen and Miller, Ezra},
     TITLE = {Gr\"obner geometry of {S}chubert polynomials},
   JOURNAL = {Ann. of Math. (2)},
  FJOURNAL = {Annals of Mathematics. Second Series},
    VOLUME = {161},
      YEAR = {2005},
    NUMBER = {3},
     PAGES = {1245--1318},
      ISSN = {0003-486X,1939-8980},
   MRCLASS = {05E15 (13C40 13F55 13P10 14M15 14N15)},
  MRNUMBER = {2180402},
MRREVIEWER = {Harry\ Tamvakis},
       DOI = {10.4007/annals.2005.161.1245},
       URL = {https://doi.org/10.4007/annals.2005.161.1245},
}

@article {KM04,
    AUTHOR = {Knutson, Allen and Miller, Ezra},
     TITLE = {Subword complexes in {C}oxeter groups},
   JOURNAL = {Adv. Math.},
  FJOURNAL = {Advances in Mathematics},
    VOLUME = {184},
      YEAR = {2004},
    NUMBER = {1},
     PAGES = {161--176},
      ISSN = {0001-8708,1090-2082},
   MRCLASS = {20F55 (05E05 05E15 05E25 13F55)},
  MRNUMBER = {2047852},
MRREVIEWER = {Jian-yi\ Shi},
       DOI = {10.1016/S0001-8708(03)00142-7},
       URL = {https://doi.org/10.1016/S0001-8708(03)00142-7},
}

@article {KY04,
    AUTHOR = {Knutson, Allen and Yong, Alexander},
     TITLE = {A formula for {$K$}-theory truncation {S}chubert calculus},
   JOURNAL = {Int. Math. Res. Not.},
  FJOURNAL = {International Mathematics Research Notices},
      YEAR = {2004},
      Note = {No. 70},
     PAGES = {3741--3756},
      ISSN = {1073-7928,1687-0247},
   MRCLASS = {14M15 (14C35 19E08)},
  MRNUMBER = {2101981},
MRREVIEWER = {Xuhua\ He},
       DOI = {10.1155/S1073792804142244},
       URL = {https://doi.org/10.1155/S1073792804142244},
}

@article {KoM05,
    AUTHOR = {Kogan, Mikhail and Miller, Ezra},
     TITLE = {Toric degeneration of {S}chubert varieties and
              {G}elfand-{T}setlin polytopes},
   JOURNAL = {Adv. Math.},
  FJOURNAL = {Advances in Mathematics},
    VOLUME = {193},
      YEAR = {2005},
    NUMBER = {1},
     PAGES = {1--17},
      ISSN = {0001-8708,1090-2082},
   MRCLASS = {14M15},
  MRNUMBER = {2132758},
MRREVIEWER = {Kiumars\ Kaveh},
       DOI = {10.1016/j.aim.2004.03.017},
       URL = {https://doi.org/10.1016/j.aim.2004.03.017},
}

@article {WY12,
    AUTHOR = {Woo, Alexander and Yong, Alexander},
     TITLE = {A {G}r\"obner basis for {K}azhdan-{L}usztig ideals},
   JOURNAL = {Amer. J. Math.},
  FJOURNAL = {American Journal of Mathematics},
    VOLUME = {134},
      YEAR = {2012},
    NUMBER = {4},
     PAGES = {1089--1137},
      ISSN = {0002-9327,1080-6377},
   MRCLASS = {14M15 (13P10)},
  MRNUMBER = {2956258},
MRREVIEWER = {Venkatramani\ Lakshmibai},
       DOI = {10.1353/ajm.2012.0031},
       URL = {https://doi.org/10.1353/ajm.2012.0031},
}

@article {KST12,
    AUTHOR = {Kirichenko, V. A. and Smirnov, E. Yu. and Timorin, V. A.},
     TITLE = {Schubert calculus and {G}elfand-{T}setlin polytopes},
   JOURNAL = {Uspekhi Mat. Nauk},
  FJOURNAL = {Uspekhi Matematicheskikh Nauk},
    VOLUME = {67},
      YEAR = {2012},
    NUMBER = {4(406)},
     PAGES = {89--128},
      ISSN = {0042-1316,2305-2872},
   MRCLASS = {14N15 (14M15 52B12)},
  MRNUMBER = {3013846},
MRREVIEWER = {\c Serban\ B\u arc\u anescu},
       DOI = {10.1070/RM2012v067n04ABEH004804},
       URL = {https://doi.org/10.1070/RM2012v067n04ABEH004804},
}

@article {PSW24,
    AUTHOR = {Pechenik, Oliver and Speyer, David E. and Weigandt, Anna},
     TITLE = {Castelnuovo-{M}umford regularity of matrix {S}chubert
              varieties},
   JOURNAL = {Selecta Math. (N.S.)},
  FJOURNAL = {Selecta Mathematica. New Series},
    VOLUME = {30},
      YEAR = {2024},
    NUMBER = {4},
     PAGES = {Paper No. 66, 44},
      ISSN = {1022-1824,1420-9020},
   MRCLASS = {05E05 (05E40 13C40 14M15)},
  MRNUMBER = {4768772},
MRREVIEWER = {Giuseppe\ Favacchio},
       DOI = {10.1007/s00029-024-00959-x},
       URL = {https://doi.org/10.1007/s00029-024-00959-x},
}

@article {HS24,
    AUTHOR = {Huang, Daoji and Striker, Jessica},
     TITLE = {A pipe dream perspective on totally symmetric
              self-complementary plane partitions},
   JOURNAL = {Forum Math. Sigma},
  FJOURNAL = {Forum of Mathematics. Sigma},
    VOLUME = {12},
      YEAR = {2024},
     PAGES = {Paper No. e17, 19},
      ISSN = {2050-5094},
   MRCLASS = {05A19},
  MRNUMBER = {4696012},
MRREVIEWER = {Sam\ Hopkins},
       DOI = {10.1017/fms.2023.131},
       URL = {https://doi.org/10.1017/fms.2023.131},
}

@article {GH23,
    AUTHOR = {Gao, Yibo and Huang, Daoji},
     TITLE = {The canonical bijection between pipe dreams and bumpless pipe
              dreams},
   JOURNAL = {Int. Math. Res. Not. IMRN},
  FJOURNAL = {International Mathematics Research Notices. IMRN},
      YEAR = {2023},
      Note = {No. 21},
     PAGES = {18629--18663},
      ISSN = {1073-7928,1687-0247},
   MRCLASS = {14M15},
  MRNUMBER = {4665632},
MRREVIEWER = {Arpita\ Nayek},
       DOI = {10.1093/imrn/rnad083},
       URL = {https://doi.org/10.1093/imrn/rnad083},
}

@article {MPP19,
    AUTHOR = {Morales, Alejandro H. and Pak, Igor and Panova, Greta},
     TITLE = {Asymptotics of principal evaluations of {S}chubert polynomials
              for layered permutations},
   JOURNAL = {Proc. Amer. Math. Soc.},
  FJOURNAL = {Proceedings of the American Mathematical Society},
    VOLUME = {147},
      YEAR = {2019},
    NUMBER = {4},
     PAGES = {1377--1389},
      ISSN = {0002-9939,1088-6826},
   MRCLASS = {05A05 (05A16 05E05 14N15)},
  MRNUMBER = {3910405},
MRREVIEWER = {Arthur\ L. B. Yang},
       DOI = {10.1090/proc/14369},
       URL = {https://doi.org/10.1090/proc/14369},
}

@article {G21,
    AUTHOR = {Gao, Yibo},
     TITLE = {Principal specializations of {S}chubert polynomials and
              pattern containment},
   JOURNAL = {European J. Combin.},
  FJOURNAL = {European Journal of Combinatorics},
    VOLUME = {94},
      YEAR = {2021},
     PAGES = {Paper No. 103291, 12},
      ISSN = {0195-6698,1095-9971},
   MRCLASS = {05E16 (05E14 14N15)},
  MRNUMBER = {4192112},
MRREVIEWER = {Artem\ A.\ Lopatin},
       DOI = {10.1016/j.ejc.2020.103291},
       URL = {https://doi.org/10.1016/j.ejc.2020.103291},
}

@misc {williamson2003mind,
  title={Mind your P and Q-symbols: Why the Kazhdan-Lusztig basis of the Hecke algebra of type A is cellular},
  author={Williamson, Geordie},
  year={2003}
}

@article{BottSamelson1958,
  author  = {Bott, Raoul and Samelson, Hans},
  title   = {Applications of the theory of Morse to symmetric spaces},
  journal = {American Journal of Mathematics},
  volume  = {80},
  number  = {4},
  pages   = {964--1029},
  year    = {1958},
  publisher = {Johns Hopkins University Press}
}

@article{Demazure1974,
  author  = {Demazure, Michel},
  title   = {D{\'e}singularisation des vari{\'e}t{\'e}s de {S}chubert g{\'e}n{\'e}ralis{\'e}es},
  journal = {Annales Scientifiques de l'{\'E}cole Normale Sup{\'e}rieure},
  series  = {4},
  volume  = {7},
  number  = {1},
  pages   = {53--88},
  year    = {1974}
}

@article{Hansen1973,
  author  = {Hansen, Hans},
  title   = {On cycles on the flag manifold},
  journal = {Matematisk Institut Preprint Series},
  Note  = {No. 11},
  pages   = {1--14},
  year    = {1973},
  publisher = {University of Copenhagen}
}

@article{Sta17b,
  author    = {Richard P. Stanley},
  title     = {Some Schubert shenanigans},
  journal   = {S\'eminaire Lotharingien de Combinatoire},
  volume    = {78B},
  pages     = {Art. 36, 12 pp.},
  year      = {2017},
  note      = {FPSAC 2017}
}

@book{Humphreys1990,
  title={Reflection Groups and Coxeter Groups},
  author={Humphreys, James E.},
  series={Cambridge Studies in Advanced Mathematics},
  volume={29},
  year={1990},
  publisher={Cambridge University Press}
}

@book{BjornerBrenti2005,
  title={Combinatorics of Coxeter Groups},
  author={Bj{\"o}rner, Anders and Brenti, Francesco},
  series={Graduate Texts in Mathematics},
  volume={231},
  year={2005},
  publisher={Springer}
}

@article{GutierresMamedeSantos2020,
  author  = {Gutierres, Gon{\c{c}}alo and Mamede, Ricardo and Santos, Jos{\'e} Luis},
  title   = {Commutation Classes of the Reduced Words for the Longest Element of {$\mathfrak{S}_n$}},
  journal = {The Electronic Journal of Combinatorics},
  volume  = {27},
  number  = {2},
  pages   = {P2.21},
  year    = {2020},
  doi     = {10.37236/9481}
}

@article{Elnitsky1997,
  author  = {Elnitsky, Serge},
  title   = {Rhombic Tilings of Polygons and Classes of Reduced Words in Coxeter Groups},
  journal = {Journal of Combinatorial Theory, Series A},
  volume  = {77},
  number  = {2},
  pages   = {193--221},
  year    = {1997},
  doi     = {10.1006/jcta.1997.2723}
}

@article{AwikBrelandCadmanErnst2024,
  author  = {Awik, Fadi and Breland, Jadyn and Cadman, Quentin and Ernst, Dana C.},
  title   = {Braid Graphs in Simply-Laced Triangle-Free Coxeter Systems Are Partial Cubes},
  journal = {European Journal of Combinatorics},
  volume  = {118},
  pages   = {103927},
  year    = {2024},
  doi     = {10.1016/j.ejc.2024.103927},
  eprint  = {2104.12318},
  archivePrefix = {arXiv},
  primaryClass  = {math.CO}
}

@article{Eti84,
  AUTHOR = {Etienne, Gwihen},
  TITLE = {Linear extensions of finite posets and a conjecture of {G}. {Kreweras} on permutations},
  JOURNAL = {Discrete Math.},
  FJOURNAL = {Discrete Mathematics},
  VOLUME = {52},
  YEAR = {1984},
  NUMBER = {1},
  PAGES = {107--111},
  ISSN = {0012-365X},
  DOI = {10.1016/0012-365X(84)90108-0},
  URL =
    {https://www.sciencedirect.com/science/article/pii/0012365X84901080},
  PUBLISHER = {Elsevier},
}

\end{document}